\documentclass[12pt]{amsart}
\usepackage{amsmath}
\usepackage{amssymb}
\usepackage{amsfonts}
\usepackage{verbatim}
\usepackage{amscd}
\usepackage{cite}
\usepackage{leftidx}
\usepackage{enumerate}
\usepackage{txfonts}
\usepackage{manfnt}
\usepackage{amscd}
\usepackage{mathpazo}
\usepackage[mathscr]{eucal}
\usepackage{hyperref}
 \DeclareMathOperator{\Zeros}{Zeros}
  \DeclareMathOperator{\zeros}{Zeros}

\DeclareMathOperator{\tor}{Tor}
 
 \DeclareMathOperator{\Res}{Res}
 \DeclareMathOperator{\gm}{\mathbb G_m}
\def\norm{{\mathrm{norm}}}

\DeclareMathOperator{\Bl}{\mathcal {B}\l}
 \DeclareMathOperator{\bl}{b\l}
\def\inv{^{-1}}

\def\sepp{{\mathrm{sepp}}}
\def\ssep{^{\mathrm{sep}}}
\DeclareMathOperator{\mgbar}{\overline\M_g}
\DeclareMathOperator{\del}{\partial}
 \DeclareMathOperator{\Hom}{Hom}
\DeclareMathOperator{\coker}{coker}
\DeclareMathOperator{\out}{out}
\DeclareMathOperator{\inn}{inn}
\DeclareMathOperator{\bD}{\mathbb D}
\def\pre{\mathrm{pre}}

\def\refp #1.{(\ref{#1})}

\newcommand{\A}{\mathcal{A}}
\newcommand{\J}{\mathcal{J}}
\newcommand{\M}{\mathcal{M}}

\newcommand{\bE}{\mathbb{E}}

\newcommand{\sE}{{\mathrm s\mathbb E}}
\newcommand{\aE}{{\leftidx{_\mathfrak a}{\mathbb E}{}}}

\newcommand{\lbl}[1]{\label{#1}}

\newcommand{\Cal}[1]{\mathcal #1}

\newcommand{\ul}[1]{\underline {#1}}

\def\sbr #1.{^{[#1]}}
\def\sfl #1.{^{\lfloor #1\rfloor}}

\newcommand{\myit}[1]{\emph{\ {#1}\ }}
\newcommand{\mymax}{{\max}}
\newcommand{\mydiamond}[4]
{\begin{matrix}&& #1&&\\ &\swarrow&&\searrow&\\ #2&&&& #3\\
&\searrow&&\swarrow&     \\&& #4&&
\end{matrix}
}
\newcommand{\mycd}[4]{
\begin{matrix}
#1&\to&#2\\\downarrow&\circlearrowleft&\downarrow\\
#3&\to&#4
\end{matrix}
}

\newcommand\twostack [2]
{\begin{matrix}#1 \\ #2\end{matrix}}
\newcommand{\sectionend}
{\[\circ\gtrdot\ggg\mathbb{\times}\lll\lessdot\circ\]}

\newcommand\mymatrix[2] {\begin{matrix} #1\\ #2\end{matrix}}

\newcommand\red{{\mathrm red}}
\newcommand\vir{{\mathrm{vir}}}
\newcommand\sep{{\mathrm{sep}}}
\newcommand\phel{{\rm{PHEL }}}
\newcommand\vtheta{{\underline{\theta}}}
\DeclareMathOperator{\Mod}{Mod}

\def\inv{^{-1}}
\def\?{{\bf{??}}}
\def\iso{\ isomorphism\ }

\newcommand{\bn}{{\ Brill-Noether\ }}

\def\Proj{\textrm{Proj}}

\def\rx{\leftidx{_\mathrm R}{X}{}}
\def\lx{{\leftidx{_\mathrm L}{X}{}}}
\def\a{{\frak a}}

\def\M{\Cal M}
\def\A{\Bbb A}

\def\C{\mathbb C}
\def\P{\mathbb P}

\def\R{\mathbb R}
\def\Z{\mathbb Z}

\def\ord{\text{\rm ord}}

\def\sym{\text{\rm Sym} }

\def\ls{\vskip.25in}

\def\Q{\mathbb Q}

\def\L{\mathcal L}
\def\rmL{\mathrm L}

\def\O{\mathcal O}

\def\Sym{\textrm{Sym}}

\def\g{\mathfrak g}
\def\J{\mathfrak J}

\def\m{\mathfrak m}

\def\1/2{\frac{1}{2}}

\def\D{\mathfrak D}

\def\I{\mathcal{ I}}

\def\im{\text{im}}
\def\simto{\stackrel{\sim}{\rightarrow}}
\def\2{{[2]}}
\def\l{\ell}
\def\nl{\newline}

\def\he{\mathcal{HE}}
\def\hebar{\bar{\mathcal{HE}}}

\def\<{\langle}
\def\>{\rangle}

\def\im{\text{im}}
\def\2{{[2]}}
\def\l{\ell}

\def\Proj{\text{Proj}}
\def\scl #1.{^{\lceil#1\rceil}}
\def\spr #1.{^{(#1)}}
\def\sbc #1.{^{\{#1\}}}

\def\subpr#1.{_{(#1)}}

\newcommand{\aphi}{\leftidx{_\mathfrak a}\phi}

\def\beq{\begin{equation*}}
\def\eeq{\end{equation*}}

\newcommand{\sing}{{\mathrm{sing}}}
\newcommand{\lf}[1]{\leftidx{_{\ \mathrm L}}{#1}{}}
\newcommand{\rt}[1]{\leftidx{_{\ \mathrm R}}{#1}{}}
\newcommand{\lfrt}[1]{\leftidx{_{\ \mathrm LR}}{#1}{}}

\def\g3{{\Gamma\spr 3.}}

\def\ggg{{\Gamma\spr 3.}}

\def\hep{hyperelliptic\ }
\def\nhep{non-hyperelliptic\ }
\def\eva{essentially very ample\ }
\def\fkb {\frak b}

\newcommand{\md}[1]{\leftidx{_{\ \mathrm M}}{#1}{}}
\def\cJ{\mathcal J}
\def\rR{\mathrm R }
\def\rL{\mathrm L }
\def\ae{\leftidx{_{\mathfrak a}}{\mathbb E}}
\def\aE{\leftidx{_{\mathfrak a}}{\mathbb E}}

\newcommand{\eqspl}[2]{
\begin{equation}\label{#1}
\begin{split}
#2\end{split}\end{equation}}
\newcommand{\eqsp}[1]{\begin{equation*}
\begin{split}#1\end{split}\end{equation*}}

\newcommand{\exseq}[3]{
0\to #1\to #2\to #3\to 0
}
\newcommand{\beginalphaenum}{
\begin{enumerate}\renewcommand{\labelenumi}{ }
\item \begin{enumerate}
}

\def\eex{\end{rm}\end{example}}
\newcommand\newsection[1]{\section{#1}\setcounter{equation}{0}
}
\newcommand\newsubsection[1]{\subsection{#1}\setcounter{equation}{0}}

\newcommand{\be}{\mathbb E}
\newcommand{\loc}{\mathrm{loc}}
\newcommand\dual{\ \check{}\ }

\newtheorem{thm}{Theorem}  [section]

\newtheorem*{thm*}{Theorem}
\newtheorem*{prop*}{Proposition}
\newtheorem{cor}[thm]{Corollary}
\newtheorem*{cor*}{Corollary}

\newtheorem{lem}[thm]{Lemma}
\newtheorem{lem*}{Lemma}
\newtheorem{claim}[thm]{Claim}
\newtheorem*{claim*}{Claim}
\newtheorem{prop}[thm]{Proposition}
\newtheorem{propdef}[thm]{Proposition-definition}
\newtheorem{defn}[thm]{Definition}
\theoremstyle{remark}

\newtheorem{rem}[thm]{Remark}
\newtheorem{crit-rem}[thm]{Critical remark}
\newtheorem{remarks}[thm]{Remarks}

\newtheorem{example}[thm]{Example}
\newtheorem*{example*}{Example}

\newtheorem*{defn*}{Definition}

\begin{document}
\title {Canonical systems and their limits on
stable curves}

\author
{Ziv Ran}
\date {\today}
\address{\tiny  {\newline Ziv Ran \newline University of California
Mathematics Department\newline Surge Facility, Big Springs Rd.
\newline Riverside CA 92521
\newline \tt{ziv.ran @ ucr.edu}}}
\subjclass{14H10,
14H51}\keywords{Canonical system, nodal curve, linear system, degeneration of curves}

\begin{abstract}
We propose an object called 'sepcanonical system' on a
 stable curve $X_0$ which is to serve as limiting object-
  distinct from other such limits introduced previously-
  for the canonical system, as a smooth curve degenerates to $X_0$.
First, for
 curves  which cannot be separated by 2 or fewer nodes (the so-called '2-inseparable' curves), 
 the sepcanonical
  system consists of the sections of the dualizing sheaf, and
  fails to be very ample iff $X_0$ is a limit of
  smooth hyperelliptic curves (such $X_0$ are called 2-inseparable hyperelliptics).
 For general, 2-separable curves $X_0$,
this assertion is false, leading us to introduce the sepcanonical system, which is a collection of linear
systems on the '2-inseparable parts' of $X_0$,
each associated to a different twisted limit of the
canonical system, where the entire collection varies smoothly
with $X_0$. To define sepcanonical
system, we must endow
the curve with extra structure called an 'azimuthal structure'. We show (Theorem \ref{azi-he-thm}) that the sepcanonical system is
 'essentially very ample' unless the curve is a tree-like
 arrangement of 2-inseparable hyperelliptics. In a subsequent paper \cite{grd} we will show
 that the latter property is equivalent to the curve being a
 limit of smooth  hyperelliptics, and will essentially give defining equation for the closure of the locus of smooth hyperelliptic curves in the moduli space of stable curves.

\end{abstract}
\thanks{arXiv.org/math.AG/1104.4747}

\maketitle
\tableofcontents

\section*{Introduction}
A big part of the geometry of a nonsingular curve revolves around its canonical system and canonical map, which is always an embedding except for
the well-understood exception of hyperelliptic curves. The
purpose of this paper is to identify and study an appropriate
'limiting object', called \emph{sepcanonical sytem}, 
of the canonical system as the curve degenerates to a stable one, an issue
that is closely related to the extrinsic geometry, especially defining equations, of the closure of the hyperelliptic locus in the moduli space of stable curves. The main results are Theorem \ref{sepcanonical-thm},
which is (equivalent to) a characterization of the limit object by data on the limit
curve, and Theorem \ref{azi-he-thm} which characterizes
limits of smooth hyperelliptic curves in terms of sepcanonical systems.  \par
The default choice for the limiting object is certainly
the canonical system itself,
i.e. the linear system associated to the dualizing sheaf, on the limiting stable curve. This choice
 will be analyzed in \S 1-3 below. What the analysis shows is that the canonical system on a stable
curve $X_0$ is a good choice of limit as long as $X_0$ is '2-inseparable'
in the sense that it cannot be disconnected by removing 2 or fewer nodes. We will show (Theorem \ref{hel-2-insep-thm}) that a 2-inseparable curve
is a limit of smooth hyperelliptic curves iff its dualizing sheaf is not
very ample. Those curves were classified long ago by Catanese \cite{catanese-gorenstein}.\par
Going beyond 2-inseparable stable curves $X_0$, it becomes less than clear
that the canonical system on $X_0$ is the correct limiting object:
for example it is no longer true
 that $X_0$ is the limit of hyperelliptics
whenever its canonical map is not birational; indeed the locus of stable curves of genus $g$ whose canonical map is not birational has
components of fixed codimension $c$ independent of $g$ contained in
the boundary (in fact, in the boundary component $\Delta_0$) of
the Deligne-Mumford  Moduli space $\mgbar$. One the other hand, the  analysis of 2-inseparable curves suggests viewing these
as atoms for the purpose of constructing the limiting object for a general curve. \par
Before proceeding to describe our limiting object, it should be mentioned here that there exist general constructions
for limit objects of linear series, such as the one published
in an Inventiones paper by
 Eisenbud and Harris (cf. \cite{eisenbud-harris}) and an unpublished one \cite{ran-dls}.
 The Eisenbud-Harris construction is extended and
  studied in much detail for the case of the canonical series, 
 especially on  curves comprised of two components intersecting in  generically positioned points,
 in an Inventiones paper
 by Esteves-Medeiros \cite{esteves-medeiros} (which also references \cite{ran-dls}). 
 But this limit linear series is not a good limiting object from our perspective, because it does not vary smoothly,
  i.e. it jumps, in families (even locally trivial ones),
and therefore does not appear to be useful for the sort of enumerative applications
that we have in mind (\cite{internodal}, see below).\par
The limiting object that we propose here is
 called the \emph{sepcanonical system}. It is essentially
 a part of the full limit series associated to the canonical
 series (for a suitable modification of the family): more specifically,
 the largest part that never jumps.
To define the sepcanonical system the curve first has to he endowed with some
additional structure, closely related to Main\`o's notion
of enrichment \cite{maino}, which we call an \emph{azimuthal structure}, and which
consists essentially of a collection of smoothing directions at separating pairs of nodes (so-called 'biseps') .
The composite object is called an \emph{azimuthal curve}, and
the parameter space for these is a certain blowup of the
moduli space, in which the locus of 2-separable curves becomes
a divisor. This notion is best behaved for curves of 'semi-compact type', meaning that
distinct separating pairs of nonseparating nodes (called 'proper biseps') are disjoint, because
then the smoothing directions may be chosen independently.
Things are more complicated for non-semi compact-type curves,
but this turns out not to be much of a problem because
we shall see that in
a number of specific ways, those curves behave non- hyperelliptically.\par
Given an azimuthal curve $X_0$, the sepcanonical system
 $|\omega_{X_0}|^\sep$ is defined as a collection of
linear systems $|\omega_{X_0}|^\sep_Y$
containing $|\omega_{X_0}|_Y$, one on each  '2-component' $Y$ of $X_0$, i.e. subcurve obtained
by blowing up the separating nodes (called 'seps') and biseps
 of $X_0$. The apparently arbitrary definition is justified,
  in part, by Theorem
\ref{sepcanonical-thm} which shows that in a versal deformation
of $X_0$, each  $|\omega_{X_0}|^\sep_Y$
occurs as the restriction on $Y$ of the
limit of a certain twist (by a linear
combination of boundary components depending on $Y$) of the relative canonical bundle. Still one may ask\begin{itemize}
 \item
 why the particular choice of twist ?
 \item in what sense do the individual $|\omega|^\sep_Y$
 for different $Y$ form
 a whole?
 \end{itemize}
 As to the first question,
the choice
of twist is dictated by by the requirement, motivated as mentioned above by enumerative
applications, that
$|\omega_{X_0}|^\sep_Y$  should be as large as possible
(in particular it should contain $|\omega_{X_0}|_Y$), but
should  never
jump as $X_0$ deforms keeping $Y$ intact: in particular
the dimension of $|\omega_{X_0}|^\sep_Y$ depends only on the combinatorics. The fact that enlarging $|\omega_{X_0}|^\sep_Y$ must come at the expense of $Y$'s neighbors severely limits
the possible twist, leading to our particular choice
(which, as mentioned above, differs from choices made previously, e.g. in \cite{eisenbud-harris} and \cite{ran-dls}).\par
The second question is answered by a result in \cite{grd}
which shows that the various $|\omega_{X_0}|^\sep_Y$ together
effectively arise from a single bundle map, the so-called azimuthal modification of the Brill-Noether map
(see below).\par
Theorem \ref{sepcanonical-thm} in turn is a consequence of
an elementary 'Residue Lemma' \ref{residue-lem} which
for certain
twists of the canonical bundle identifies  the sections on the special fibre which
 lift to the general fibre.
\par The main justification for the sepcanonical system is
provided by results in \cite{grd} (Theorem 7.1 and special cases including 4.8, 5.10-5.12), based on the
present paper as well as on \cite{internodal}. These results
extend Theorem \ref{hel-2-insep-thm} to the general case and characterizes limits of smooth hyperelliptic curves.
As a step towards this result, we prove here a characterization (Theorem \ref{azi-he-thm}) of stable curves whose sepcanonical system is not 'essentially' very ample (e.g. non-birational)
on some 2-component, extending Catanese's
result for the canonical system: namely, those curves are
trees of hyperelliptic 2-inseparables joined at Weierstrass points and hyperelliptic divisors. These are precisely
the curves which are an admissible double cover of a rational tree.
\par 
The paper \cite{grd} may in fact be viewed as a \emph{global} (over Moduli) and \emph{quantitative} (or
enumerative)  counterpart to the local (over the base) and qualitative results of this paper. It constructs
a map of vector bundles 
\[\leftidx_{\a}{\phi}{}:\leftidx_{\a}{\be}{}\to\Lambda\]
(over the degree-2 relative Hilbert scheme of a given family), which is a birational modification
of the usual Brill-Noether evaluation map of the relative canonical bundle, 
where the geometry associated to $\leftidx_{\a}{\phi}{}$ is in a suitable sense that of the
sepcanonical system acting on degree-2 schemes.
\par
An alternative (enumeratively-qualified) approach to limits of hyperelliptic curves (and more general
1-dimensional linear systems) focuses on the hyperelliptic pencil and its limits (as maps to rational trees), rather than the canonical system. This is based on the
theory of admissible covers due to Harris-Mumford \cite{har-mum} and its generalizations such as the notion of
relative maps developed by Faber and Pandharipande \cite{faber-pandh}.
As mentioned above, another approach to the limit of the canonical series is due to Esteves-Medeiros \cite{esteves-medeiros}, based on
\cite{eisenbud-harris}. Some of the lemmas and propositions in \S 2-4 are similar in substance, if perhaps not in point of view,
 to results published previously.
Also, our notion of curve with azimuthal structure is
related
to the notion of
'enriched curve' developed in an unpublished Harvard thesis by L. Main\`o \cite{maino}
 (which in turn references our unpublished preprint \cite{ran-dls}).
In general, the notion of enriched curve,
in which biseps play no special role, is different
from that of (regular) azimuthal curve, in which they do;
however, the two notions coincide for semicompact-type curves.

\par
This paper was originally part of \cite{grd}, but is being published separately due to length considerations.\par
I thank Gwoho Liu for helpful comments, and David Eisenbud for
sending a copy of Main\`o's thesis.\ls\ls

\noindent
{\bf{Stable curves}}\ \  In this paper, we work over $\C$ 
and all curves are nodal unless otherwise stated.
A \emph{semistable }  curve is a nodal curve $X$ whose dualizing sheaf $\omega_X$
has nonnegative   degree on each irreducible component; equivalently, the desingularization
of each rational component of $X$ contains at least 2 node preimages; a semistable curve is \emph{stable} if the degree
  of $\omega_X$ is positive on each irreducible component.  A \emph{stable
pair} or \emph{stable pointed curve} is a pair consisting of a semistable
curve $X$ and a smooth point $p$ on it so that
$\omega_X(p)$ has positive degree on each irreducible component.
 Note if $\theta$ is a
separating node on a stable curve $X$, then
\[(X,\theta)=(\lf{X}, \lf{\theta})
\bigcup\limits_{\lf{\theta}\to\theta\leftarrow\rt{\theta}} (\rt{X},\rt{\theta})\] with each pair $(_*X,_*\theta)$ stable
(as pair). We call the $(_*X, _*\theta)$ the
\emph{left and right parts} of $X$ with respect to $\theta$,
usually denoted $\lf{X}(\theta), \rt{X}(\theta)$.\par
\newsection{Inseparables and base of the Canonical system}\lbl{insep}
The results of this section are not new (see e.g.
 \cite{catanese-gorenstein}) and are included merely for
 completeness and because we have a different viewpoint
 that will continue later in the paper. The theme is that the canonical 
 system's behavior on length-1 schemes, i.e.
its base locus, is related to separating nodes.
This theme is extended in the next section to relate behavior on
length-2 schemes, i.e. separation, to separating pairs of nodes.
 We begin by considering
curves
without separating nodes.\par A semistable curve is said to be
\emph{separable} if it has a separating node, \emph{inseparable}
otherwise. The \emph{separation} of a nodal curve $X$ is the blowup
$X^\sep$ of $X$ at all its separating nodes, called \emph{seps} for short. Clearly the connected
components of $X^\sep$ are the maximal connected inseparable
subcurves of $X$, called its \emph{inseparable components}. Any
irreducible component is contained in a unique inseparable
component. The \emph{separation tree} of a (connected nodal) curve
$X$ is the graph, necessarily a tree, whose vertices are the
inseparable components of $X$ and whose edges are the separating
nodes of $X$. More generally, we may associate a separation tree to
any collection of separating nodes on $X$. Separability is closely
related to freeness of the canonical system, as the following Lemma
begins to show:
\begin{lem}\lbl{insep-is-free}
 The canonical system of an inseparable connected semistable
 curve of genus $>0$ is free (of base points).
\end{lem}
\begin{proof}
 We may assume $X$ is singular, and let $X'\to X$ be the blowup
 of a (nonseparating) node $\theta$, with node preimages $p_1, p_2$. Then $X'$
 is a connected nodal curve of genus $g-1$ and $h^0(\omega_{X'}(p_1+p_2))=g
 >h^0(\omega_{X'})$. From the definition of dualizing sheaf it follows
 that $\omega_{X'}(p_1+p_2)$ has a section nonvanishing at both $p_1, p_2$
 hence $\omega_X$ is free at $\theta$. Thus $\omega_X$ is free at all
 (automatically nonseparating) nodes. It follows easily that any base point
 of $\omega_X$ must occur on a smooth rational components $C$,
each of  which must contain $r\geq 2$ nodes. Consider the restriction map
 \[\rho: H^0(\omega_X)\to H^0(\omega_X|_C)=H^0(\omega_C(r)).\]
  Its kernel may be identified with the differentials on the complementary subcurve to $C$, which has genus $g-r+1$.
  Therefore $\rho$ has rank $r-1$, hence is surjective, so $\omega_X$ has no base points on $C$.
Therefore $\omega_X$ is free.
\end{proof}
A smooth rational inseparable component of a nodal curve $X$ is called a
a \emph{separating line} of $X$ (see \cite{abel-gor}, Definition 4.7) and $X$ is said to be a
\emph{comb} if it has a separating line. The following result is due to Catanese (\cite{catanese-gorenstein}, Thm. D),
who uses a different terminology.
\begin{cor}[Catanese]
The base locus of the canonical system on a semistable curve
is the union of the separating lines and the separating nodes.
\end{cor}
\begin{proof}
To begin with, it is elementary from Riemann-Roch
that the base locus contains
the separating lines and the separating nodes. For the converse, we
use induction on the number of irreducible components of the
curve $X$, assumed connected and of genus $>0$.
If $X$ has a separating line, we can conclude easily by applying induction to
the remaining inseparable components. Else,
let $Y$ be
an inseparable component that constitutes an end vertex of the separation
tree of $X$. Then $X=Y\cup_pZ$ with $Z$ connected
and semistable and without separating lines, and
$H^0(\omega_Y)=H^0(\omega_Y(p))$
again by Riemann-Roch. By Lemma \ref{insep-is-free},
the base  locus of $\omega_X$ on $Y$
is exactly $p$.
As $H^0(\omega_X)=H^0(\omega_Y)\oplus H^0(\omega_Z)$,
the proof is completed by applying induction to $Z$.
\end{proof}

A pair of smooth points on a nodal curve is said to be \emph{separated}
if they lie on different inseparable components.
\begin{cor}\lbl{separated-are-indep} A separated pair of smooth points, not on any separating lines, of a connected
stable curve, impose independent conditions on the canonical
system.\end{cor}

\begin{proof}
If $p$ is a smooth point on a non-separating-line inseparable component $C$
of $X$, there
is a by Lemma \ref{insep-is-free} a differential on $C$ nonzero at
$p$, and using the natural map $H^0(\omega_C)\to H^0(\omega_X)$, this may be extended to a dualizing differential on $X$
 that is zero on $\overline{X-C}$.
\end{proof}
 An inseparable semistable curve $X$ is  said to be \emph{ PHEL}
 (acronym for pseudo-hyperelliptic), if the canonical system $|\omega_X|$ fails to define
en embedding. This includes all semistable curves or arithmetic genus 0 or
1. A stable pair $(X,p)$ with $X$ inseparable is said to be
 PHEL if $|\omega_X|$ is ramified at $p$. This terminology
 is temporary because we will see later
 that PHEL is equivalent to another notion,
 to be defined below even for possibly separated curves,
  that we will call 'hyperelliptic'.
\begin{lem} Suppose $(X,p)$ is \phel with $X$ inseparable
 and singular.\par
(i) If $X$ is irreducible of arithmetic genus $>1$, the canonical system of $X$ defines
a 2:1 map to a smooth rational curve, ramified on $p$. \par
(ii) If $X$ is reducible and $p$ lies on an irreducible component $C$ of positive arithmetic genus,
then $C$ meets the rest of $X$ in
precisely 2 points $p_1,p_2$; if $C$ has genus $>1$ then $C$ is \phel
and its $g^1_2$ contains $2p$ and $p_1+p_2$.
\end{lem}
\begin{proof} (i) is a special case of Proposition \ref{insep-hel}.
(ii) follows from the fact that for any Cartier divisor  $D$ of degree $>2$, $\omega_C(D)$
is very ample. Of course by inseparability, $C$ cannot meet the rest of $X$
in $<2$ points.
\end{proof}
\newsection{2-inseparables and separation by the Canonical system}\lbl{2-insep}
The purpose of this section is to develop some lemmas needed in the proof of the
main result.
These lemmas concern the relation between 
point separation (very ampleness) properties of the canonical system and some of its twists and
\emph{ separating binodes} of the curve.
They also yield short proofs of results of Catanese \cite{catanese-gorenstein} and others
classifying the 2-inseparable
nodal (more generally, Gorenstein) curves whose dualizing sheaf is not very ample, and
characterization of the limits of
hyperelliptic curves (Proposition \ref{hel-2-insep-thm}). \par
 First some definitions and remarks that will be used throughout the paper.
 \begin{propdef}
\begin{enumerate}\item
A pair of nodes on a nodal curve $X$ is called a \emph{binode}. A binode
$(\theta_1, \theta_2)$ is \emph{properly separating}
 or a \emph{bisep}, for short, if each $\theta_i$
is nonseparating but $X-\{\theta_1, \theta_2\}$ is disconnected, hence
necessarily has exactly 2 connected components whose closures are- at this point, arbitrarily- designated
the \emph{ left} and \emph{right sides} of $(\theta_1, \theta_2)$,
denoted $\lf{X}(\theta_1, \theta_2), \rt{X}(\theta_1, \theta_2)$;
in particular,
$\vtheta=(\theta_1, \theta_2)$ itself will have a left and right preimages denoted $\lf{\vtheta}\subset\lf{X}(\vtheta), 
\rt{\vtheta}\subset\rt{X}(\vtheta)$.
A subcurve of $X$ of the form $\lf{X}(\theta), \rt{X}(\theta)$ for a sep or bisep $\theta$ is simply called a \emph{side} of $X$. We will call
a subset $\theta$ of $X$ a *-sep if it is either a sep or bisep.
\item
A nodal curve is \emph{2-inseparable} if it is not disconnected
by removal of any 2 or fewer nodes.
\item The \emph{2-separation} of $X$ is the blowing-up of all seps and biseps. The connected components of the 2-separation
    are called \emph{2-components} of $X$.

\end{enumerate}
\end{propdef}
\begin{lem}
A 2-component is always inseparable.
 \end{lem}
 \begin{proof} Suppose
a  2-component $Y$ is separable and let $\theta$ be a sep of $Y$, necessarily not a sep of $X$.
If $\vtheta=(\theta_1, \theta_2)$ is a bisep of $X$ contained in $Y$ and $\vtheta\subset \lf{Y}(\theta)$ or
$\vtheta\subset \rt{Y}(\theta)$, $\theta$ is already a sep on $X$. Otherwise, $(\theta_1, \theta)$ is a bisep on $X$,
contradiction. A sep of $X$ contained in $Y$ leads to a similar contradiction.\end{proof}
\begin{lem}\lbl{higher-twist-very-ample}
Let $X$ be a 2-inseparable nodal curve and $\a$ an effective Cartier divisor of degree $\geq 3$ on it. Then $\omega_X(\a)$ is very ample.
\end{lem}
\begin{proof}
Set $a=\deg(\a), g=p_a(X)$. As $h^0(\a)=g+a-1$ , we must prove
\[h^0(\omega_X(\a-\frak b))=g+a-3\] for every length-2 subscheme
$\frak b$ of $X$. Assume first that $\frak b$ is regular.
 Set $L=\O(\frak b-\a)$. This has degree $<0$ overall and
 $\leq 2$ on each subcurve. By Riemann
Roch and Serre duality,
 what has to be shown is $h^0(L)=0$. If not, let $s\in H^0(L)$
 and $C$ the union of all irreducible components
 of $X$ on which $s$ is not
 identically zero. Because $\deg(L)<0$ while $\deg(L|_C)\geq 0$, there must be a connected component $Y$ of the complementary curve $\overline{X-C}$ such that $\deg(L|_Y)<0$, hence $s$ vanishes identically on $Y$. Because $|Y\cap C|\geq 3$, $s|_C$ is a regular section of $L|_C$ with $\geq 3$ zeros, which contradicts $\deg(L|_C)\leq 2$.\par
This leaves only the case where $\frak b$ is tangent to a branch at a node $\theta$ of $X$. Let $b:X'\to X$ be the blowing up of $\theta$ and $\a'=b^*(\a)$, of degree $a$. Then
$X'$ is inseparable of arithmetic genus $g-1$, and
\[h^0(\omega_X( \a-\frak b))=h^0(\omega_{X'}(\a'-p)\]
with  $p$ the preimage
of $\theta$ corresponding to $\frak b$. What must be proven here is
\[H^0(\O_{X'}(p-\a'))=0.\] This proof is similar to the above, using inseparability of $X'$.
\end{proof}
This result admits a useful partial extension to the inseparable
case. First a definition.
\begin{defn}\lbl{rel-insep-defn}A nodal curve $X$
 is said to be \emph{2-inseparable} (resp. inseparable)
 \emph{relative} to a subset (or divisor or subscheme...)
 $\a$ if given a any bisep (resp.
 sep) $\theta$ of $X$, (the support of) $\a$ meets both $\lf{X}(\theta)$ and
 $\rt{X}(\theta)$. Equivalent terminology is that$(X, \a)$
 is 2-inseparable or inseparable.
 \end{defn}
A useful remark about this notion
 is the following.
\begin{lem}
Let $Y$ be a 2-component of $X$ and $\vtheta$ a bisep of $X$
contained in $Y$. Then $Y$ is inseparable
 and 2-inseparable relative to
$\vtheta$.
\end{lem}
\begin{proof}
We have seen that
 $Y$ is inseparable. Let $\vtheta'$ be a 2-sep of $Y$.
If $\vtheta\subset \lf{Y}(\vtheta')$ or $\vtheta\subset \rt{Y}(\vtheta')$, then $\vtheta'$ would be a bisep of $X$
itself, against our hypothesis that $Y$ is a 2-component.
\end{proof}
The following Lemma on nonpositive bundles will prove useful.
\begin{lem}\lbl{deg-0-lem}
Let $L$ be a nontrivial line bundle of degree
$\leq 0$ on an inseparable curve
$X$, such that $L$ has degree $\leq 2$ on any subcurve of $X$
and degree $\leq 1$ on either side of any bisep.
Then $H^0(L)=0$.
\end{lem}
\begin{proof}
If $s$ is a nonzero section vanishing somewhere,
let $Y_1, Y_2\subset X$ be the union of all irreducible
components on which $s$ is (resp. is not) identically
zero. Since $s$ is a regular section of $L$ over $Y_2$,
it cannot have more than 3 zeros there.
Also, $Y_1\neq\emptyset$. Therefore
$Y_1$ and $Y_2$ meet in at most, hence exactly, 2 points
which constitute a bisep. By the assumption on biseps, $L$ has
degree at most 1 on $Y_2$, so again we have a contradiction.

\end{proof}
Now we can prove the desired partial extension of Lemma \ref{higher-twist-very-ample}:
\begin{lem}\lbl{omega+3-on-2-sep-lem}
Let $X$ be an inseparable nodal curve and $\a$ an effective
smoothly supported divisor of degree $\geq 3$.
Assume $X$ is 2-inseparable relative to $\a$.
Then $\omega_X(\a)$ is very ample.
\end{lem}
\begin{proof}
We follow the outline and notations of the proof of Lemma \ref{higher-twist-very-ample}.
We may assume $X$ is 2-separable. Consider first
the case where $\fkb$ is regular. Then $L=\O(\fkb-\a)$
satisfies the hypotheses of Lemma \ref{deg-0-lem}, hence
$H^0(L)=0$, as desired.
Now to complete the proof, following the argument of the Lemma \ref{higher-twist-very-ample},
it will suffice to consider the case where $\frak b$
is tangent to the left side at a node $\theta_1$
belonging to a bisep $(\theta_1, \theta_2)$.
Using the above
notations, we need to show $H^0(\O_X'(p-\a'))=0$
where $p\in\lf{X'}(\theta_2)$. Now $X'$
is inseparable, and clearly $L= \O_X'(p-\a')$ satisfies
the hypotheses of Lemma \ref{deg-0-lem}, so we are done.
\end{proof}

According to Caporaso, an \emph{$r$-vine} is
a nodal curve of genus $g=r-1$ of the form
\[ (\P^1, p_1,...,p_r)\bigcup\limits_{\twostack{p_i\leftrightarrow q_i}
{i=1,...,r}}
 (\P^1, q_1,...,q_r).\]
 An $r$-vine is said to be an \emph{interlace}
 if
the  pair of  pointed $\P^1$'s are isomorphic (this notion,
but not the term, already occur in Catanese's article
\cite{catanese-gorenstein}). Note
that an interlace $X$ admits a 2:1 morphism
$\eta$ to $\P^1$, ramified precisely over
the nodes and by Riemann-Hurwitz we have
\[\omega_X\sim(g-1)\eta\inv(pt).\]
Because we can always assume $(p_1,p_2,p_3)=(q_1,q_2,q_3)=(0,1,\infty)$,
the parameter space for $r$-vines is an open subset of $(\C^{g-2})^2$ and that of $r$-interlaces is the diagonal.\par

An irreducible nodal curve is said to be \emph{hyperelliptic}
 if it is obtained
from a smooth hyperelliptic curve of genus $>1$ by identifying some pairs belonging to the
$g^1_2$, or obtained from a irreducible nodal or smooth curve  of arithmetic genus 1 by identifying some
pairs in a single $g^1_2$.
The following is an initial characterization of 2-inseparable
hyperelliptics; the definitive result is Proposition \ref{insep-hel} below.
\begin{lem}\lbl{hypell-or-interlace}
Let $p,q$ be smooth, not necessarily distinct points on a 2-inseparable curve $X$.
Then $p+q$ imposes independent conditions on $\omega_X$ unless either\par
(i) $X$ is irreducible hyperelliptic with $p+q\in g^1_2$; or\par
(ii) $X$ is an interlace and $p+q=\eta\inv(pt)$.
\end{lem}
\begin{proof}
By Riemann-Roch, $h^0(\O(p+q))=2$.
Assume $p,q$ are in a single irreducible component $C$.
Suppose $X$ is reducible and let $Y_1$ be a connected component of the complementary curve $Y=\overline{X-C}$. The $Y_1$ meets $C$ in a set $\a$ of at least 3 points, and we have an exact sequence
\[\exseq{\omega_{Y_1}}{\omega_{C\cup Y_1}}{\omega_C(\a)}\]
As $H^1(\omega_{Y_1})\simeq H^1(\omega_{C\cup Y_1})$ (both curves being connected), the map
$H^0(\omega_{C\cup Y_1})\to H^0(\omega_C(\a))$ is surjective.
But by the Lemma above, $p+q$ clearly imposes 2 conditions on
$\omega_C(\a)$, therefore also on $\omega_X$, contradiction.
Therefore $Y_1$ cannot exist and $X$ is irreducible.
As $|p+q|$ induces a degree-2 pencil on the normalization of $X$, or on a genus-1 partial normalization if $X$ is rational,
it is easy to see
that $X$ is irreducible hyperelliptic. \par
If $p\in C, q\in D$ are in different irreducible
components, connectedness
ensures injectivity of
the restriction maps \[H^0(\O_X(p+q))\to H^0(\O_C(p)),
H^0(\O_X(p+q))\to H^0(\O_D(q)).\] Therefore
 \[h^0(\O_C(p)),
h^0(\O_D(q))\geq 2,\] so $C,D$ are smooth
rational and there is a 2:1 map $C\cup D\to \P^1$.
We claim $C\cup D=X$. If not,
let $Y$ be a connected component of the complementary curve. Because $Y$ meets $C\cup D$ in at least 3 points, we may assume $Y$ meets $C$ in at least 2 points, therefore $\omega_C(Y\cap C)$ has nonnegative degree, hence is free. Then from surjectivity of
\[ H^0(\omega_{C\cup Y})\to H^0(\omega_C(C\cap Y))\]
we conclude there is a differential on $C\cup Y$ not vanishing at $p$, hence a section of $\omega_X$ zero on $D$ and nonvanishing at $p$. On the other hand
because $X$ is inseparable, $\omega_X$ is free so there is a section of $\omega_X$ nonvanishing at $q$. Therefore $p,q$ impose independent conditions on $\omega_X$, contradiction.\par
Therefore $X=C\cup D$ is an interlace.

\end{proof} The above argument proves more generally the following
\begin{lem}
Let $X,C,Y$ be 2-inseparable nodal curves with $X=C\cup Y$.
Then\par (i) the restriction map $H^0(\omega_X)\to H^0(\omega_C(C\cap Y))$ is surjective;\par
(ii) if $|C\cap Y|\geq 3$, then $|\omega_X|$ is very ample on the interior of $C$, i.e. $C\setminus C\cap Y$.
\end{lem}
Next we extend the result of Lemma \ref{higher-twist-very-ample} to the case of twisting by
a degree-2 divisor:
\begin{lem}\lbl{omega+2}
Let $X$ be a 2-inseparable nodal curve and $\a$ an effective Cartier divisor of degree $2$ on it. Then $\omega_X(\a)$ is very ample off $\a$, i.e. separates every length-2 scheme except $\a$, unless $(X,\a)$ is hyperelliptic.
\end{lem}
\begin{proof}
Actually, Proposition \ref{insep-hel} below- which is proved
independently of Lemma \ref{omega+2}- already implies $X$ is hyperelliptic and this could be used to show $\a$ is a
hyperelliptic divisor. But we take a different tack instead,
and follow the outline and notations of the proof of Lemma \ref{higher-twist-very-ample}. 
First let  $\frak b$ be any regular
length-2 scheme other than
$\a$ and $L=\O(\frak b-\a)$, this time of degree 0. If $h^0(L)=0$, $\omega_X(\a)$ separates $\frak b$. If $s\neq 0\in H^0(L)$ vanishes nowhere, then
$\O(\a)=\O(\frak b)$ and $\a$ is a hyperelliptic divisor by
Lemma \ref{hypell-or-interlace}. If $s$ vanishes somewhere,
it is identically zero on some component, and an argument
as in the proof of Lemma \ref{higher-twist-very-ample} applies. This settles the case $\frak b$ regular. The case $\frak b$
irregular is identical to that of Lemma \ref{hypell-or-interlace}.
\end{proof}
We can partly extend this to the inseparable case:
\begin{lem}\lbl{omega+2-insep-lem}
Let $Y$ be an inseparable curve, $p\neq q$ smooth points
such that $Y$ is 2-inseparable relative to $p+q$.
Then $\omega_Y(p+q)$ is very ample off $p+q$,
unless $Y$ is 2-inseparable and $(Y, p+q)$ is hyperelliptic.
\end{lem}
\begin{proof}
If $Y$ is 2-inseparable, the above Lemma applies.
Else, $p,q$ belong to different 2-components
 $Y_1, Y_2$. It will suffice prove that  $|\omega(p+q)|$,
separates every length-2 scheme $\fkb$ except
at $p+q$.
We will focus on the $\fkb$ regular case, as the irregular
one is handled as in Lemma \ref{omega+3-on-2-sep-lem}.
If $\O(\fkb)$,
hence $L=\O(\fkb-p-q)$ has degree $\leq 1$ on each
component, Lemma \ref{deg-0-lem} applies, so we may assume
$\O(\fkb)$ has degree 2 on a unique 2-component $Z$.
We may assume $p\not\in Z$.
Let $\theta$ be a bisep on $Z$ with $p\in\lf{Y}(\theta),
Z\subset\rt{Y}(\theta)$. Then clearly any section
$s\in H^0(L)$ vanishes on $\lf{Y}(\theta)$, hence
on $\rt{\theta}\subset Z$. If $q\in Z$, then $L$ has degree
1 on $Z$ but $s$ has 2 zeros there, so $s$ vanishes on $Z$,
hence on all of $Y$. If $q\not\in Y$, there exists
another bisep $\theta'\neq \theta$ with $q\in\lf{Y}(\theta'),
Z\subset \rt{Y}(\theta')$. Then $s$ has 3 zeros on $Y$, hence vanishes s above.
\end{proof}
The following result due to Catanese \cite{catanese-gorenstein} classifies the hyperelliptic 2-inseparable curves. Given the
foregoing Lemmas, the proof is short and we will include it.
\begin{prop}[Catanese]\lbl{insep-hel}
Let $X$ be a 2-inseparable nodal curve of arithmetic genus $g>1$, such that $\omega_X$ is not very ample; then
$X$ is either\par
(i)  irreducible hyperelliptic, or\par
 (ii) an interlace.\par  In either case, the canonical system $|\omega_X|$ yields a  2:1 morphism onto a rational normal curve in $\P^{g-1}$.
\end{prop}
\begin{proof} Assume $X$ is neither irreducible hyperelliptic nor an interlace.
We need to show every length-2 scheme imposes
independent conditions,
i.e. is embedded by $|\omega_X|$.  The case of a scheme supported at
smooth points was considered above.
The case of a scheme $p+q$ where $p$ is a node and $q\neq p$
 is elementary.
Next, consider the case of a length-2
 scheme $\zeta$
supported at a node $p$. If $X'$ is the blowup of $p$ with node preimages
$p_1, p_2$ then $\zeta$ dependent for $\omega_X$ means $p_1+p_2$
dependent for $\omega_{X'}$ (this is true whether $\zeta$
is regular or not).
Note $X'$ is inseparable. If $X'$ is 2-insparable then Lemma \ref{hypell-or-interlace}
applies. Else, let $\vtheta$ be a separating binode of $X'$.
If $\vtheta$ is properly contained in a polyseparator $\Theta$
(see \S\ref{polysep-sec} below), then some connected component of $(X')^\Theta$ contains neither $p_1$ nor $p_2$, which makes $X$ 2-separable.
Therefore $\vtheta$ is already a maximal polyseparator, hence by Lemma \ref{n-gon}, its sides
$\lf{X}', \rt{X}'$ are both inseparable and again by 2-inseparability of $X$, each side contains precisely one of $p_1, p_2$. Because
$\omega_{\lf{X}'}, \omega_{\rt{X}'}$ are both free and
$H^0(\omega_{\lf{X}'}), H^0(\omega_{\rt{X}'})\subset H^0(\omega_{X'})$ , this implies
$p_1, p_2$ impose 2 conditions on $\omega_{X'}$.
\end{proof}
The following  characterization of the limits of smooth hyperelliptic curves among 2-inseparables is probably not new
(compare for instance \cite{caporaso-imrn}, \S 5.2) but is included
because it follows easily from the foregoing discussion. 
\begin{prop}\label{hel-2-insep-thm}
Let $X$ be a 2-inseparable nodal curve of arithmetic genus $g>1$.
Then the following are equivalent:\par (i) $X$ is a limit of smooth
hyperelliptic curves.\par (ii) The dualizing sheaf $\omega_X$ is not
very ample.\par (iii) $X$ is irreducible hyperelliptic or an interlace.
\end{prop}
\begin{proof}
The equivalence of (ii) and (iii) is just
Catanese's result (Proposition \ref{insep-hel}).
That (i) implies (ii) is trivial: if a smooth hyperelliptic curve $X$
specializes to a stable curve $X_0$, then a general point $p\in X_0$
will be contained in a length-2 scheme imposing 1 condition on $\omega_{X_0}$, so $\omega_{X_0}$ is not even birationally very ample.
It remains to prove  that (iii) implies (i). This is basically
a folklore dimension
counting argument that we will give in the case of an interlace,
as that of an irreducible curve is similar. Thus let $X_0$ be an
interlace, which we may assume corresponds to 2 identical $(g+1)$-pointed
$\P^1$'s of the form
\[(\P^1_1, 0,1,\infty, p_1,...,p_{g-2})\simeq(\P^1_2, 0,1,\infty, q_1,...,q_{g-2}).\] Let $X/B$ be a versal deformation of $X_0$.
Let $B_1\subset B$ denote the locus of $(g+1)$-nodal
curves, which may be identified locally with $(\C^{g-2})^2\times T$ for a polydisk $T$,
with the $2(g-2)$ parameters corresponding to $p_1,...,q_{g-2}$.
On $B$, $B_1$ is locally defined by $(g+1)$ equations corresponding
to the nodes, say $u_1,...,u_{g+1}$..
Let $p$ be a general point of $\P^1$, and let $p'\in\P^1_1\subset  X_0$, $p"\in\P^1_2\subset  X_0$
be the points corresponding to $p$. Then in a neighborhood of $(p',p")$
we may identify $X_{B_1}^2$ with $(\C^{g-1})^2\times T$.\par
Now let $E$ be the Hodge bundle $\pi_*(\omega_{X/B})$ pulled back
to $X^2_B$ and consider the evaluation map
\[\phi:E\to p_1^*(\omega_{X/B})\oplus p_2^*(\omega_{X/B}).\]
At $(p',p")$, $\phi$ has rank 1 hence in a neighborhood of $(p',p")$
on $X^2_B$, the degeneracy locus $D$ of $\phi$ is defined by
$g-1$ equations, say $f_1,...,f_{g-1}$, not necessarily forming a regular sequence. On the other hand,
consider the restriction of $\phi$ over $X_{B_1}^2$, identified
with $(\C^{g-1})^2\times T$. There, the degeneracy locus
of $\phi$ corresponds to hyperelliptic $(g+1)$-nodal curves plus
member of the $g^1_2$, and by Proposition \ref{insep-hel} this
can be identified with $\C^{g-1}\times T$ sitting diagonally in
$(\C^{g-1})^2\times T$, hence is still defined by a regular sequence
of length $g-1$. It follows that $f_1,...,f_{g-1}$ forms a regular sequence $\mod (u_1,...,u_{g+1})$, i.e. $f_1,...,f_{g-1}, u_1,...,u_{g+1}$
together form a regular sequence, hence  no $u_i$ is a zero-divisor
modulo $f_1,...,f_{g-1}$ i.e. $D$ has no component contained in
any boundary divisor corresponding to a node of $X_0$.
Therefore $(p',p")$ is a limit of hyperelliptic divisors
 on a smooth hyperelliptic curve.

\end{proof}
\begin{rem} This remark is folklore, only peripheral to our results.
Via the theory of admissible covers \cite{har-mum} and its generalizations
as in \cite{eisenbud-harris} and \cite{faber-pandh}, the limits of hyperelliptics can also be characterized among the stable 2-inseparables as those admitting a 2:1 map to $\P^1$. 
The 2-inseparability is what makes the target  a $\P^1$ rather than a general rational tree.
This result is at least implicit  in \cite{har-mum}, Sec. 4 and the more powerful \cite{faber-pandh},
but does not appear to have been written down explicitly. In any event,
a drawback of  this characterization is that it does not readily yield defining equations for the hyperelliptic locus. By contrast, the above result readily yields such equations,
more precisely, equations for the locus of pairs (hyperelliptic curve, fibre of hyperelliptic map) as a degeneracy locus
(always among 2-inseparables).\qed
\end{rem}

Now as soon as one leaves the realm of 2-inseparable curves,
birational non-very ampleness of the canonical system is no longer
a good notion of hyperellipticity: as the example below
shows, the canonical
map can have different degrees (1 or 2) over different components of its image, and such curves cannot be limits of smooth hyperelliptics. This
 is precisely the motivation for developing the notion
 of the 'sepcanonical system' the will occupy us in subsequent
 sections. One of the main results of the paper \cite{grd} will show that the limits
of hyperelliptics are precisely the curves whose sepcanonical
system is not  birationally very ample.

\begin{example}
Let $\vtheta$ be a separating binode on a nodal curve $X$,
so that the pair \nl \mbox{$(\lf{X}(\vtheta), \lf{\theta_1}+\lf{\theta_2})$}
is a smooth hyperelliptic pair while $\rt{X}$ is non-
hyperelliptic. Then $\omega_X$ is not
essentially very ample but $X$
cannot be a fibre of a family with smooth total space
 and general fibre a smooth hyperelliptic curve $X'$.
In
an elementary fashion this can be seen from the fact that the locus of hyperelliptic divisors (pairs not separated by the
canonical system)
 is a divisor on $(X')\spr 2.$, which would extend to a divisor
 on the relative Hilbert scheme of the family (which
 can be assumed to have only quotient singularities);
however, this locus meets the component
$(\rt{X})\spr 2.$ of the Hilbert scheme
$X\sbr 2.$ in the isolated point
$\theta_1+\theta_2$.\par
A similar example can be made with $\vtheta$ replaced by a separating node, hyperelliptic only on one side.

\end{example}

\newsection{Polyseparators}\lbl{polysep-sec}
This section is essentially trivial in nature. Its purpose of this section is to collect for future reference some elementary,
 and probably well-known (compare \cite{caporaso-viviani1}, \cite{caporaso-viviani2}), combinatorial remarks about nondisjoint collections of separating binodes on a curve. We begin with
a few definitions.
\begin{defn}\lbl{polysep-defn} Fix an inseparable nodal curve $X$.
\begin{enumerate}\item
A polyseparator of degree $n$ on $X$  is a collection
 of nodes $(\theta_1,...,\theta_n), n\geq 2$
 with the property that
any distinct pair of nodes $\theta, \theta'\in\Theta$ is a
separating binode.
\item $\Theta$ is a \emph{proper} polyseparator if
$n\geq 3$.
\item A binode or polyseparator is \emph{maximal} if it is not
properly contained in a polyseparator.
\item We associate to a polyseparator $\Theta$ a
graph $G(\Theta)$ having as vertices the connected components of the
separation or blowup $X^\Theta=b_\Theta X$ and as edge-set $\Theta$.\end{enumerate}
\end{defn}\begin{defn}\label{semicompact-defn}\begin{enumerate}
\item
On a general nodal curve $X$, we define a polyseparator as a
polyseparator on one of its inseparable components.
\item $X$ is said to be of \emph{semicompact type}
(or 'polyunseparated') if it has no proper polyseparators.
\end{enumerate}
\end{defn}
\begin{lem}\lbl{n-gon}
Let $\Theta$ be a polyseparator of degree $n$ on an inseparable curve $X$. Then\quad (i) $G(\Theta)$ is a simple $n$-gon and, with suitable notation,
\[X=\bigcup\limits_{i=1}^n\rt{X}(\theta_i, \theta_{i+1}),
\theta_{n+1}:=\theta_1;\]
(ii) given a node $\theta\not\in\Theta$, $\Theta\cup\{\theta\}$ is a polyseparator iff $\theta$ is a separating node on the unique
$\rt{X}(\theta_i, \theta_{i+1})$ containing it..

\end{lem}
\begin{proof}
Induction on $n\geq 2$. With (i) being obvious for $n=2$, we first prove (i) implies (ii) for given $n$. Notations as above, suppose $\theta$ is a node on $\rt{X}(\theta_1, \theta_2)$ separating it in 2 connected components $L,R$. Because $X$ is inseparable, precisely one of $\theta_1, \theta_2$ is on each of $L,R$ and we may assume $\theta_1\in L, \theta_2\in R$. But then clearly the separation
\[X^{\theta, \theta_i}=[R\cup\bigcup\limits_{j=1}^{i-1}\rt{X}(\theta_j, \theta_{j+1})]\coprod [\bigcup\limits_{j=i}^{n}\rt{X}(\theta_j, \theta_{j+1})\cup L]\] so that $(\theta, \theta_i)$ is a separating binode. Conversely, suppose $\theta$ is a nonseparating node on $\rt{X}(\theta_1, \theta_2)$.
Then $X^{\theta, \theta_1, \theta_2}$ has just 2 connected components and becomes connected when the two preimages of $\theta_2$ are identified, i.e. $X^{\theta,\theta_1}$ is connected, so $(\theta,\theta_1)$ is nonseparating and
$\Theta\cup\{\theta\}$  is not a polyseparator.\par
To complete the proof it suffices to show that (i) and (ii) for given $n$ implies (i) for $n+1$. So suppose $\Theta=(\theta_1, ...
,\theta_{n+1})$ is a polyseparator of degree $n+1$. Then $(\theta_1,...,\theta_n)$ is a polyseparator and we may assume the above notations apply.
 Using (ii) and shifting cyclically,  we may assume $\theta_{n+1}$ is a separating node on $\rt{X}(\theta_n,\theta_1)$ and inseparability of $X$ again implies that $\theta_{n+1}$ has exactly one of $\theta_n, \theta_1$ or either of its sides in
$\rt{X}(\theta_n,\theta_1)$. This shows $G(\Theta)$ is an $(n+1)$-gon.\end{proof}
A notation as above is called a \emph{cyclic arrangement} for the polyseparator $\Theta$.

\begin{lem}
Any polyseparator on an inseparable curve is contained in a unique maximal polyseparator.
\end{lem}
\begin{proof}
Let $\Theta$ be a polyseparator and use a cyclic arrangement as in the previous lemma. Let $M(\Theta)$ be the union of $\Theta$ and the separating nodes on $\rt{X}(\theta_i, \theta_{i+1})$.
By the Lemma, any polyseparator containing $\Theta$ is contained in $M(\Theta)$. Therefore it suffices to prove $M(\Theta)$ is a polyseparator. The proof is by induction on
$m(\Theta):=|M(\Theta)\setminus \Theta|$. If $m(\Theta)>0$, pick any $\theta\in M(\Theta)\setminus \Theta$ and let
$\Theta'=\Theta\cup \{\theta\}$. To finish the induction it suffices to prove $M(\Theta')=M(\Theta)$. With no loss of generality, we may assume $\Theta'=(\theta_1,...,\theta_{n+1})$ is a cyclic arrangement as above.
If $\alpha\in M(\Theta')$,
we may assume $\alpha$ is a separating node on $\rt{X}(\theta_n,\theta_{n+1})$. By inseparability of $X$ again,
$\alpha$ has exactly one of $\theta_n,\theta_{n+1}$ on either side in $\rt{X}(\theta_n,\theta_{n+1})$, which makes it
a separating node on $\rt{X}(\theta_n,\theta_{1})$, so
$\alpha\in M(\Theta)$. Conversely, if $\alpha\in M(\Theta)$,
we may assume $\alpha$ is a separating node on $Y=\rt{X}(\theta_n,\theta_{1})$ with $\theta_n\in\lf{Y}(\alpha), \theta_1\in\rt{Y}(\alpha)$.
But then if $\theta_{n+1}\in\lf{Y}(\alpha)$, then $\alpha$ is a separating node on $\rt{X}(\theta_{n+1}, \theta_1)$, while if
$\theta_n\in\rt{Y}(\alpha)$, then $\alpha$ is a separating node on $\rt{X}(\theta_n,\theta_{n+1})$. Therefore in either case
$\alpha\in M(\Theta')$.
\end{proof}
\begin{lem}
Any two maximal polyseparators on an inseparable nodal curve are disjoint.
\end{lem}
\begin{proof}
Let $\Theta=(\theta_1,...,\theta_n)$ be a maximal polyseparator, cyclically arranged, and $\theta$ another node,
say $\theta\in\rt{X}(\theta_i, \theta_{i+1})$. If $\theta$ separates $\rt{X}(\theta_i, \theta_{i+1})$, then $\theta\in \Theta$ by maximality.
If not, then clearly $X^{\theta_j,\theta}$ is connected for all $j$ (we can go around the circle starting from $\theta_j$ and avoiding $\theta$). Therefore, $\theta_j$ cannot be in any polyseparator not contained in $\Theta$, so cannot be in any other maximal polyseparator. This proves the required disjointness.

\end{proof}
We summarize the above results as follows.
\begin{prop}
Let $X$ be an inseparable nodal curve. Then there is a collection
of disjoint sets of nodes on $X$ called maximal polyseparators, with the
property that a given binode $\vtheta$ is properly
separating iff $\vtheta$ is contained in some  maximal polyseparator.
\end{prop}
A (oriented) bisep $\vtheta$ is said to be \emph{adjacent} (or \emph{right-adjacent} ) on $X$ if $\rt{X}(\vtheta)$ is inseparable. This terminology is justified by the following result,
which follows easily from the above discussion:
\begin{cor}
$(\theta_1, \theta_2)$ is adjacent on $X$ iff the unique maximal polyseparator containing it has the form, in cyclic arrangement, $\Theta=(\theta_1, \theta_2,...,\theta_m), m\geq 2$.
\end{cor}
\begin{lem}
In a curve of semi-compact type, any two separating binodes lie entirely to one side of each other.
\end{lem}
\begin{proof}
Let $\vtheta\neq \vtheta'$ be separating binodes of $X$, and assume for contradiction that $\theta'_1\in\lf{X}(\vtheta),
\theta'_2\in\rt{X}(\vtheta)$. Then $b_{\theta'_1}\lf{X}(\vtheta),  b_{\theta'_2}\rt{X}(\vtheta)$
are connected and their images  in $b_{\vtheta'}X$ clearly meet, so $b_{\vtheta'}X$ is connected, contradiction.
\end{proof}
The following two results collect some elementary properties
of graphs of curves and of curves built up from 2-inseparable
hyperelliptics, respectively. All the required proofs are easy
enough to warrant omission.
\begin{propdef}\lbl{2sep-tree-gen-case-defn}
Let $X$ be a nodal curve.
\begin{enumerate}\item
The \emph{ 2-separation tree} $G_2(X)$ is
the dual graph having as edges the separating nodes and all maximal (i.e. not contained in a proper polyseparator)
properly separating binodes, and as vertices
the connected components of the blowup of all the
latter nodes and binodes.
\item This graph is a tree.
 \end{enumerate}
\end{propdef}
For a nodal curve $X$, we will denote by $S(X)$ the set of all its
seps and biseps. Recall from the beginning of the previous section that a A  {2-separation component} (or 2-component) of
$X$ is by definition a
connected component of the blowup of $X$ in S(X), and that a  2-component is always inseparable.
\begin{propdef}
\begin{enumerate}
\item Any 2-component $Y$ is endowed with a collection of smooth singletons or 'unimarks' coming from separating
    nodes of $X$, and smooth pairs or 'bimarks' coming
    from separating binodes. Unless otherwise mentioned, these are always oriented with $Y$ on the left.
\item A sep $\theta$ is said to be \emph{locally left-hyperelliptic} if for the 2-component $Y$ to the left of $\theta$, $(Y, \lf{\theta})$ is hyperelliptic;
ditto for \emph{locally right-} and \emph{bilaterally} (left + right) \emph{hyperelliptic}; ditto for bisep in place of sep.
\item a locally left or right-hyperelliptic bisep is automatically maximal.
\item A stable curve $X$ is said to be \emph{hyperelliptic } if it is of semicompact type
and every edge of $G_2(X)$ is locally bilaterally hyperelliptic; given $X$, it is is hyperelliptic iff every
sep and bisep is locally bilaterally hyperelliptic.
\item For a subset $\Theta\subset S(X)$, $X$ is said to be \emph{ hyperelliptic relative to $\Theta$}
    if every $\theta\in\Theta$ is locally bilaterally hyperelliptic.
\item A sep $\theta$ is said to be \emph{left hyperelliptic} if
$(\lf{X}(\theta), \lf{\theta})$ is hyperelliptic; ditto for right and bilateral; ditto for bisep;
\item $\theta$ is left hyperelliptic iff it is locally left hyperelliptic and every sep and bisep strictly to the left of $\theta$ is locally bilaterally hyperelliptic.
\item $\theta$ is said to be \emph{left hyperelliptic
 relative to $\Theta$}  it is locally left hyperelliptic and every sep and bisep in $\Theta$ strictly to the left of $\theta$ is locally bilaterally hyperelliptic;

\item A connected nodal curve $X$ is said to be \emph{hyperelliptic} if its
stable model is.
\end{enumerate}
\end{propdef}
\newsection{Residue Lemma}
Here we prove an elementary result  (Lemma \ref{residue-lem})about sections of
twists of the relative canonical system in a family, that underlies
the definition of the sepcanonical system
to be given in \S \ref{sepcanonical} . To
first  introduce the Lemma somewhat informally,
consider a degeneration of curves $X/B$
with smooth total space over the disk, and
let $D$ be a sum of components of the special fibre, so
that $\omega_{X/B}\otimes\O(D)$ is a 'twisted form'
of $\omega_{X/B}$, agreeing with it over the generic fibre.
Let
$s_i:B\to X$ be sections disjoint from $D$, $i=1,...,n$. Then the bundle
$\omega_{X_0}\otimes(\O(D)\otimes\O_{X_0})(\sum s_i(0))$
is a limit of $\omega_{X_b}(\sum s_i(b))$ on the general fibre. What is the condition that a section $\alpha_0$ of the former
bundle lift to a section $\alpha_b$ of the latter bundle ?
An obvious necessary condition is $\sum \Res_{s_i(0)}(\alpha_0)=0$ (because the analogous sum for
$\alpha_b$ vanishes by the Residue Theorem).
The Lemma below says that under some mild hypotheses, this
condition is also sufficient.
\par First some terminology. For a family of curves $\pi:X\to B$,
an \emph{\'etale multisection} is an effective divisor $\sigma\subset X$ that is finite flat and \'etale over $B$
and disjoint from the singular locus of $X/B$.
$B$ may be a point, in which case an \'etale multisection
is just a reduced divisor disjoint from the singular locus (recall that we are working over $\C$!).
For such $\sigma$ and an effecive divisor $A$ disjoint
 from $\sigma$ we have, for any $n\geq 1$, the residue map
\[\Res:\pi_*(\omega_{X/B}(A+n\sigma))\to\pi_*(\O_\sigma).\]
Composing with the trace map, we get a 'sum of residues'
map
\[\sum_\sigma \Res:\pi_*(\omega_{X/B}(A+n\sigma))\to\O_B.\]
When $B$ is local, we will identify $\pi_*$ with $H^0$.

\begin{lem}[Residue Lemma]\lbl{residue-lem}
 Suppose given\begin{itemize}\item$\pi : X\to B$, a
proper flat family of nodal curves of arithmetic genus $g$,
with irreducible generic fibre, over an integral local scheme;
\item
$B_0\subset B$, a closed subscheme (the 'boundary')
with $\pi_0: X_0=X\times_{B}B_0\to B_0$ the corresponding subfamily;\item pairwise
disjoint \'etale multisections $\sigma_i, i=1,...,m$
of $X/B$, considered as relative divisors;
\item
an effective Cartier divisor $D$ on $ X$, disjoint from $\sum\sigma_i$
 and from the generic fibre, such that, setting
 $ \O_{X_0}(D_0)=\O_X(D)\otimes
 \O_{X_0}$,
 we have that the direct image sheaf
 $\pi_{0*}(\omega_{X_0/B_0}(D_0))$ is free of rank $g$ and
 compatible with base-change
 over $B_0$ ('constant sections hypothesis');\item
where $\sigma_{i,0}=\sigma_i\cap X_0$,
an element $\alpha_0\in H^0(X_0, \omega_{X_0/B_0}(D_0+\sum n_i\sigma_{i,0})), n_i\geq 1$.
\end{itemize}
 Then:
\par (i) $\alpha_0$ is a sum $\sum\limits_{i=0}^m\alpha_{0,i}$
 of sections \[\alpha_{0,i}\in
H^0(X_0,  \omega_{X_0/B_0}(D+ n_i\sigma_{i,0}))
\subset  H^0(X_0,  \omega_{X_0/B_0}(D+ \sum\limits_jn_j\sigma_{j,0}))\]iff
$\alpha_0$ satisfies the
\ul{residue conditions}: \eqspl{res-cond}
{\sum\limits_{\sigma_{i,0}}\Res(\alpha_0)=0\in\O_{B_0}, i=1,...,m.}
Moreover  the set of sections satisfying the residue
 conditions \eqref{res-cond} is a free $\O_{B_0}$- module of rank
\nl\mbox{$g+\sum\limits_{i=1}^m n_i\deg(\sigma_{i,0})-m$.}
\par
(ii) If the residue conditions \eqref{res-cond} are satisfied, then
$\alpha_0$ lifts to an element $\alpha$ of
\nl$H^0(\omega_{X/B}(D+\sum n_i\sigma_i))$ satisfying
\[\sum\limits_{\sigma_i}\Res(\alpha)=0\in\O_B, i=1,...,n\]

\end{lem}
\begin{remarks}
(i) Because $\sum\sigma_i$ and $D$ are disjoint, $X$ admits
an open affine cover \[(U_0:=X\setminus\sum\sigma_i, U_1),\]
where $U_1$ is an affine neighborhood of $\sum\sigma_i$
contained in $X\setminus D$. We can use this cover to
compute cohomology and over $U_0\cap U_1$, naturally identify
$\omega_{X/B}$ and $\omega_{X/B}(D)$ and compute residues
for the latter sheaf using this identification.\par
(ii) In applications, $D$ will usually be a sum of boundary
components, i.e. components of $X_{B_0}$.\par 
(iii) Our main application  of the Residue Lemma is Theorem \ref{sepcanonical-thm}
below.
\end{remarks}

\begin{proof}[Proof of Lemma](i):
Because the constant sections condition is open
and $B$ is local, the direct image
$\pi_*(\omega_{X/B}(D))$ is free of rank $g$
and compatible with base-change. Hence $R^1\pi_*(\omega_{X/B}(D))$ is free of rank 1
and compatible with base-change, more precisely the map
$R^1\pi_*(\omega_{X/B})\to R^1\pi_*(\omega_{X/B}(D))$ is
 an isomorphism and remains one after any base-change.
 Let $A$ be the set of sections $\alpha_0$ expressible as a sum
 $\sum\alpha_{0,i}$ as above, and let $B$ be the set of sections $\alpha_0$ satisfying the residue conditions \ref{res-cond}.
Now for each $i$, consider the long cohomology sequence of
\[\exseq{\omega_{X_0/B_0}(D_0)}{ \omega_{X_0/B_0}(D_0+ n_i\sigma_{i,0})}{\omega_{X_0/B_0}(D_0)\otimes\O_{ n_i\sigma_{i,0}}
( n_i\sigma_{i,0})}.\] The $D_0$ twist is trivial on the third term because $D$ is disjoint from $\sigma_i$.
This shows that   \nl$ H^0(\omega_{X_0/B_0}(D_0+ n_i\sigma_{i,0}))$ maps surjectively to the set of
'polar parts' in $H^0(\omega_{X_0/B_0}(D_0)\otimes \O_{ n_i\sigma_{i,0}}
( n_i\sigma_{i,0}))$ satisfying the $i$th residue condition
 in \ref{res-cond}. This implies firstly that $A\subseteq B$, i.e. 'only if' holds. Then, considering the maps
\[A\subseteq B\subseteq H^0(\omega_{X_0/B_0}(D_0+ \sum n_i\sigma_{i,0}))
\to \bigoplus\limits_i H^0( \O_{ n_i\sigma_{i,0}}
(\sum n_i\sigma_{i,0}))\] it follows that $A$, hence $B$, map onto the submodule $C$ of
$H^0( \O_{ n_i\sigma_{i,0}}
(\sum n_i\sigma_{i,0}))$ consisting of polar parts satisfying
all the residue conditions \ref{res-cond} for $i=1,...,m$, which is
obviously free and cofree of corank $m$. Because we have an exact sequence
\[\exseq{H^0(\omega_{X_0/B_0}(D_0))}{A}{C}\]
and likewise for $B$ in place of $A$, it follows that $A=B$ is free
of the required dimension.

\par (ii)
By (i), it suffices to consider
the case $m=1$ and we will omit the index. Consider the exact diagram
\small:\eqsp{
\begin{matrix}
H^0(\omega(D))&\to&H^0(\omega(D+n\sigma))&\to&H^0(\omega(n\sigma)
\otimes\O_{n\sigma}&\to& H^1(\omega(D))\\
\downarrow&&\downarrow&&\downarrow&&\downarrow\\
H^0(\omega_{X_0/B_0}(D_0))&\to&H^0(\omega_{X_0/B_0}(D_0+n\sigma_0))&\to&
\omega(n\sigma)\otimes\O_{n\sigma_0}&\to
& H^1(\omega_{X_0/B_0}(D_0))
\end{matrix}
}
\normalsize
Here the rightmost horizontal maps are just residue. In fact the entire
rightmost square is purely local and it it easy to check that
the kernel of the upper right horizontal  map-- which is a surjection
 of locally free sheaves-- goes surjectively to that
to the lower right horizontal map. This proves our assertion.
\end{proof}
\begin{rem}
The residue condition is automatically satisified on an
irreducible fibre but not on a reducible one.
\end{rem}\rm
\begin{example} Here is an illustration of how the Residue Lemma is applied.
Let $X_0=\lf{X_0}\cup_{\theta_1}\md{X_0}\cup_{\theta_2}\rt{X_0}$ be a 3-component 2-sep nodal curve, varying in a 2-parameter smoothing family $X/B$, with divisors $\del_1, \del_2$ where
$\theta_1, \theta_2$ persist, respectively. Let $\lf{D}=\lf{X}(\theta_1)$, a divisor on $X$. Then
\[\omega_{X/B}(D)|_{\md{X}}=\omega_{\md{X}}(2\rt{\theta_1}+\lf{\theta_2}).\]
We claim that a
necessary and sufficient condition for a section $\alpha$
of the latter sheaf to extend
locally to a section of $\pi_*(\omega_{X/B}(D))$ is that $\alpha$ should have zero residue, i.e. trivial pole, at $\lf{\theta_2}$. This of course implies $\alpha$ has zero residue at $\rt{\theta_1}$ as well.\par
We will prove sufficiency of the condition as necessity is similar and simpler.
  To this end, note first that by freeness of $\omega_{\lf{X_0}}$ at $\lf{\theta_1}$, $\alpha$ extends to a section of $\omega_{X/B}|_{\lf{X_0}\cup\md{X_0}} $.
Then by the Residue Lemma, this extends to a section of \[\omega_{\lf{X}(\theta_2)}(D+\lf{\theta_2})=\omega_{X/B}(D)
|_{\lf{X}(\theta_2)}\] over $\del_2$ with zero residue on $\lf{\theta_2}$. Then we glue this to any section of \[\omega_{X/B}|_{\rt{X}(\theta_2)}=\omega_{X/B}(D)|_{\rt{X}(\theta_2)}\]
(which automatically has zero residue on $\rt{\theta_2}$)
and use the Residue
Lemma again to get a section of $\omega_{X/B}$.\qed
\end{example}
The foregoing example will be generalized  in Theorem \ref{sepcanonical-thm} below.
\newsection{Sepcanonical system}\lbl{sepcanonical}
The purpose of this section is first to
define intrinsically an object on a reducible curve $X_0$, viz. a collection of linear systems on its 2-components, collectively
called the sepcanonical system of $X_0$.
 This apparently arbitrary notion will be redeemed by Theorem \ref{sepcanonical-thm}, where it will
  be shown to coincide with a certain collection of limits
of the canonical system, i.e. specializations of its twists by certain
components of the special fibre. This construction is nontrivial
 when and only when  the special fibre is 2-separable.
 The sepcanonical system depends on some additional data on the curve called azimuth (essentially, a smoothing direction at each bisep).
\begin{defn}\lbl{azimuth-defn}
\begin{enumerate}\item
If $(p_1, p_2)$ is a pair of distinct smooth points on a curve
$X$, the \emph{azimuth space} at $(p_1, p_2)$ is
$\P(T_{p_1}X\oplus T_{p_2}X)$; an \emph{azimuth} $\zeta$ at $(p_1, p_2)$ is an element of the azimuth space.
An azimuth is
singular if it equals $[1,0]$ or $[0,1]$, regular otherwise. An azimuthal pair $((p_1, p_2), \zeta)$ consists of a point-pair
 $ (p_1, p_2)$ plus an azimuth $\zeta$,
as above. The \emph{azimuthal condition} corresponding to an azimuth $\zeta$ is the condition on
on $H^0(\omega_X)$ specifying that the value of a differential  at $(p_1, p_2)$ should be in $\zeta$.  
\item
For a bisep  (properly separating binode)
$\vtheta=(\theta_1, \theta_2)$, a left
(right) azimuth is an azimuth on its left (right) preimage , i.e. an element
 $\lf{\zeta}\in \P(\lf{\psi}_{\theta_1}\oplus \lf{\psi}_{\theta_2})$ (resp.
 $\rt{\zeta}\in \P(\rt{\psi}_{\theta_1}\oplus \rt{\psi}_{\theta_2})$), where
 $\lf{\psi}_{\theta_i}=T^*_{\lf{\theta}_i}\lf{X}(\vtheta)$
 etc .
\item
For a bisep
$\vtheta=(\theta_1, \theta_2)$,
a \emph{middle azimuth} is an element \[\md{\zeta}\in\P(\lf{\psi}_{\theta_1}
\otimes\rt{\psi}_{\theta_1}\oplus \lf{\psi}_{\theta_2}\otimes
\rt{\psi}_{\theta_2}).\]
$\md{\zeta}$ is \emph{regular} if $\md{\zeta}=[a,b], a, b\neq 0$.
An \emph{azimuth} is a triple
$\zeta_\vtheta=(\lf{\zeta}, \md{\zeta}, \rt{\zeta})$ consisting
of a left, middle and right azimuth.
$\zeta_\vtheta$ is said to be  \emph{
(internally) compatible}
 if \[\md{\zeta}=\lf{\zeta}\rt{\zeta}\]
 where the multiplication refers to the obvious pairing 
 \[(\lf{\psi}_{\theta_1}\oplus \lf{\psi}_{\theta_2})\times (\rt{\psi}_{\theta_1}\oplus \rt{\psi}_{\theta_2})
 \to \lf{\psi}_{\theta_1}
 \otimes\rt{\psi}_{\theta_1}\oplus \lf{\psi}_{\theta_2}\otimes
 \rt{\psi}_{\theta_2} \]
\item
An \emph{azimuthal} (resp. \emph{middle azimuthal} curve)
 is a nodal curve together with
a choice of internally compatible azimuth (resp. middle azimuth) at each separating binode.
\end{enumerate}
\end{defn}
\begin{defn}
 An azimuthal \emph{marking} $\xi=(p.;(q.,q'.))$ on a curve
 $Y$ consists of:\par
(i) a collection of smooth points
(unimarks) $p_i\in Y$, together with a designation of each
as either 'co-hyperelliptic', with attached
 'multiplicity' $n(\xi, p)=2$, or 'non co-hyperelliptic',
  with attached multiplicity $n(\xi,p)=3$; \par (ii)
 a collection of smooth pairs (bimarks) $(q_i, q'_i)$,
together with a designation of some of them as 'co-hyperelliptic', and a choice of regular azimuth
$\zeta(q_i, q'_i)$ (only) at those.

\end{defn}
\begin{remarks}\begin{enumerate}
\item The co-hyperelliptic property of a uni- or bimark  and
the 'multiplicities' above are,
from $Y$'s perspective, completely arbitrary,
and in particular is unrelated
to being a W-point or hyperelliptic divisor on $Y$ itself.
In applications, there will be another curve attached to $Y$ at those points, and the data will be related to hyperellipticity of the other curve.

\item In an azimuth $\zeta$, any two of $\lf{\zeta},
\md{\zeta}, \rt{\zeta}$, if not both singular,
uniquely determine a third so that the triple is
a compatible azimuth.
\item
If $p_1\neq p_2$ and $p_1+p_2$ belongs to the unique
$g^1_2$ on a hyperelliptic curve $X$,
there is uniquely determined on $p_1+p_2$ a \emph{hyperelliptic azimuth}, automatically regular.
\item In particular, given a regular middle azimuth $\md{\zeta}$ on
a right-hyperelliptic bisep $\vtheta$, there a uniquely
determined left azimuth on $\vtheta$, compatible
with $\md{\zeta}$ and the hyperelliptic right azimuth.
\item
Given a bisep $\vtheta=(\theta_1, \theta_2)$ as above on a curve
$X_0$ and a smoothing $X/B$ of it, then
an element $v$ of the projectivized normal space to
$\del_{\theta_1}\cap\del_{\theta_2}$ at 0 induces a
middle azimuth $\md{\zeta}$ at $\vtheta$; $\md{\zeta}$
is regular iff $v$ corresponds to an infinitesimal smoothing
 at both $\theta_1$ and $\theta_2$.
\end{enumerate}
\end{remarks}

Now the sepcanonical system associated to a middle- azimuthal curve $X$, which we are going to define, will consist of
\begin{itemize}
\item a \emph{sepcanonical
bundle} $\omega_X^{\sep}$, a line bundle
 on the 2-separation of $X$, i.e. the \emph{disjoint} union $\coprod X_i$ of the 2-components;
\item
a linear subsystem
 $|\omega_X|^{\sep}\subset |\omega_X^{\sep}|$,
 called the \emph{sepcanonical system}.
 \end{itemize}
 More generally, we will define a $\Theta$-sepcanonical system for any collection $\Theta$ of seps and biseps so that the ordinary sepcanonical system corresponds to the case where
 $\Theta$ consists of \emph{all} seps and biseps.
 \begin{defn}\lbl{septwist-order-def}
 Let $\theta$ be an oriented *-sep (either sep or bisep) on a nodal curve $X$, $\Theta$ a collection of seps and middle-azimuthal biseps.
 \begin{itemize}\item
 $\theta$ is said to be \emph{right-hyperellptic} relative to $\Theta$ if\par
 (i) the pair $ (\rt{X}(\theta), \rt{\theta})$ is hyperelliptic;
 \par (ii) every element of $(\Theta\setminus\theta)\cap\rt{X}(\theta)$ is  hyperelliptic as a sep or middle-azimuthal bisep on $\rt{X}(\theta)$.
\item if $\vtheta\in\Theta$ is a right-hyperelliptic bisep relative to $\Theta$, the induced left azimuth on $\vtheta$ is by definition the one induced by the given middle azimuth at $\vtheta$ and the hyperelliptic azimuth on $\rt{\theta}$.
\end{itemize}
 \end{defn}
 \begin{defn}\lbl{induced-azi-marking-defn}
 Let $(X, \xi_0)$ be a stable marked curve,
  $\Theta$ a collection of seps $\theta_i$ and middle-azimuthal biseps $(\vtheta_i, \md{\zeta}(\vtheta_i)$,
  $Y$ a component of the separation $X^\Theta$.\par
Then the \emph{induced azimuthal marking} $\xi(X, \Theta)$ on $Y$ is defined as
comprised of the following unimarks and bimarks. Assume all seps and biseps of $\Theta$ are oriented
 to have $Y$ on their left.\par
(i) Unimarks: those inherited from $\xi_0$,
plus  unimarks $p=\lf{\theta}_i, \theta_i\in\Theta$, designated as co-hyperelliptic iff $\theta_i$ is right-hyperelliptic relative to $\Theta$; the induced multiplicity
is denoted $\rt{n}(X, \Theta, p)$ or $\rt{n}(X,p)$
if $\Theta$ contains all *-seps on $X$.\par
(ii) Bimarks: those inherited from $\xi_0$,
plus
bimarks $\{(q_i, q'_i)=\lf{\vtheta}_i\}$,
$\vtheta_i\in\Theta$, designated
as co-hyperelliptic iff $\vtheta_i$ is right-hyperelliptic
relative to $\Theta$, and in that case endowed with the induced azimuth.
\end{defn}

\begin{defn}\lbl{sepcanonical-defn}
Let $Y$ be a curve with
an azimuthal marking $\xi=( p., (q., q'.), \zeta.)$.
Then define the following items on $Y$:
\begin{itemize}
\item The \ul{sepcanonical twist}:
\eqspl{}{
\tau(Y, \xi)=
\sum\limits_{\textrm{marks\ }p}n(\xi,p)p+\sum\limits_{\textrm{bimarks\ }(q_i, q'_i)} 2(q_i+q'_i)
}
\item
The \ul{sepcanonical bundle}:
\eqspl{}{
\omega(Y, \xi)=\omega_Y(\tau(X,Y) )
}
\item The \ul{sepcanonical system} 
$|\omega|^\sep(Y,\xi)\subset |\omega(Y, \xi)|$ is the linear subsystem defined by the
following conditions:
\begin{enumerate}\item
Residue conditions:
\[\Res_{p_i}(\alpha)=0, \forall i;\]
\[\Res_{q_i}(\alpha)+\Res_{q'_i}(\alpha)=0, \forall i.\]
\item Azimuthal conditions corresponding to the given
azimuth at all co-hyperelliptic bimarks.
\end{enumerate}
\item When $\xi=\xi(X,\Theta)$ and $Y$ is a component of $X^\Theta$,
    the same objects will be denoted by $|\omega_X|^\sep_\Theta|_Y$ etc.
    and the collection of these
    for different components $Y$ is called the sepcanonical system of $X$
    associated to $\Theta$, denoted $|\omega_X|^\sep_\Theta$.
\item If $\a$ is a smoothly supported effective divisor
of degree $\geq 2$ on $X$, we define $\omega_X(\a)^\sep$
and $|\omega_X(\a)|^\sep$ analogously, pretending all
seps and biseps are non-hyperelliptic and imposing
residue conditions only.
\item  When $\Theta={\mathrm{Edge}}(G_2(X))$, $\Theta$
will be omitted from the notation
and the corresponding system $|\omega_X|^\sep$ will
 be called the (absolute)
    sepcanonical system of (the middle-azimuthal curve)
 $X$.
\end{itemize}
 \end{defn}
\begin{remarks}(i) Note the natural inclusion
 \[\omega_X^\sep|_Y\subset\omega_X|_Y(2\sum\lf{\theta}_i+
 \sum (q_i+q'_i))\]\par
(ii) It follows from Lemma \ref{residue-lem} that an element
 of the sepcanonical system on $Y$ is a sum of sections of
 $\omega_Y(\rt{n}(\theta_i)\lf{\theta_i})$ and
 $\omega_Y(2(q_i+q'_i))$ for the various $i$.
 \par (iii) The definition of
 $\omega_X(\a)^\sep$
and $|\omega_X(\a)|^\sep$ is certainly not the 'correct'
one in
general, but it is good enough for the few cases we need.
\par
(iv) The sepcanonical system $|\omega_X|^\sep$ is highly non-local
in the sense that its restriction on one 2-component $Y$
depends on the nature of other 2-components (as well as
on the azimuthal structure). For example, we allow a pole of order 3 at
$p=\lf{\theta}$ on $Y$, for a sep $\theta$, when 
$p$ is not co-hyperelliptic, i.e. when 
the other  (right) side $Z$ of $\theta$ in $X$ is non-hyperelliptic at $\rt{\theta}$
because $\omega_Z$ will then have a section with a simple zero at $\rt{\theta}$ to match with
a differential on $Y$ with pole of order 3 at $p$.
\par (v) The rationale for the above Residue conditions is that each $p_i$ (resp. $(q_i, q'_i)$) is
supposed to become $\lf{\theta}$ (resp. $\lf{\vtheta}$) for a sep $\theta$ (resp. bisep $\vtheta$)
on a curve $X$ containing $Y$. In this case the residue conditions hold for any differential $\alpha$
that is the restriction of a differential on $X$.
 \end{remarks}

\begin{rem}
We can extend the definition of sepcanonical system
 to the case where $X$ is semistable. (This is not
 essential, because a given family of curves can always
 be replaced with one with only stable fibres, albeit at
 the cost of making the total space singular.)
Then the stable model $\bar X$ of $X$ is obtained by
contracting a number of 'bridges', i.e. maximal
rational chains each contracting to a node $p$ of $\bar X$.
The bridge is said to be separating or not depending on $p$.
We firstly define $|\omega_X|^\sep$ to coincide with
the pullback of $|\omega_{\bar X}|^\sep$ over all components not contained in bridges. Next,
we will give a recipe  for extending this over the
bridges. The recipe may appear arbitrary, but it will be
justified in \cite{grd} in the discussion of modified \bn maps.
Thus consider a component $C$ of a bridge $B$. Then $C$
is naturally a 2-pointed $\P^1$ and may be identified
with $(\P^1, 0,\infty)$. If $B$ is a separating bridge, with sides $(\lf{ X}, \lf{p}), (\rt{ X}, \rt{p})$,
we define $|({\omega_X}^\sep)_{C}|$ to be the linear system
generated by the following differentials where $x$ is an affine coordinate
on $\A^1$ with a zero on $C\cap\lf{X}$ and a pole on
$C\cap \rt{X}$ (also given in terms of homogeneous polynomials where $x=X_1/X_0$, as subsystem of the appropriate $\O(n)$, identifying $\omega_{\P^1}=\O(2\infty)$):\begin{enumerate}
\item $dx, dx/x^2; [X_0^2, X_1^2]$,\qquad\qquad\qquad if $(\lf{X}, \lf{p}), (\rt{X}, \rt{p})$ both hyperelliptic;
    \item $xdx, dx, dx/x^2; [X_0^3, X_0^2X_1, X_1^3]$,\qquad\qquad if only $(\lf{X}, \lf{p})$ hyperelliptic;
        \item $dx, dx/x^2, dx/x^3; [X_0^3, X_0X_1^2, X_1^3]$,\qquad\qquad if only $(\rt{X}, \rt{p})$ hyperelliptic;
    \item $xdx, dx, dx/x^2, dx/x^3; [X_0^4, X_0^3X_1, X_0X_1^3, X_1^4]$,\qquad if neither side hyperelliptic
\end{enumerate} (these define respectively a double line with two branch points at $(0,\infty)$, a  plane cubic
with (cusp,flex) (resp. (flex,cusp) ) at $(0,\infty)$ and a quartic in $\P^3$ with flexes at $(0,\infty)$.
Note that the linear systems in question are uniquely determined
by their degree (as specified by the $\lf{n}(\theta), \rt{n}(\theta)$ functions) plus
the residue condition. Thus, the above apparently arbitrary definition is forced.

If $C$ is a component of a nonseparating bridge $B$, with end points
$\lf{p}, \rt{p}\in X$, we define $|({\omega_X}^\sep)_{C}|$ analogous
to case (i) above if $\lf{p}, \rt{p}$ is a hyperelliptic divisor on
the complementary curve $X'=\overline{X-B}$; and analogous to case
(iv) above, otherwise.\par
We note that with this definition,
it remains that $|\omega_X|^\sep$ defines a $g^1_2$
on each 2-component, when $X$ is semistable hyperelliptic.
\end{rem}

\begin{rem}Some of our  inductive arguments involve
the case of a 'shoot' $C$,
i.e. a nonsingular rational curve meeting the rest of $X$
in a single point $p$. Such a shoot, if part of a larger
semistable
curve $X'\supset X$, will also occur on a bridge on some other subcurve
$X"\subset X'$, hence dealt with as above.\end{rem}

The following Theorem identifies the sepcanonical system as the limit value of
 the direct image of a certain twist of the relative canonical bundle. It should be mentioned that for the Eisenbud-Harris limit series associated to the canonical, some 
 analogous results, especially for curves comprised of two components meeting in
 points in generic position, were obtained by Esteves-Medeiros
 (\cite{esteves-medeiros}, \S 5).\par
 Now the rationale for our particular choice of twist and limit is that it is the 'best of the worst'; namely, its restriction on a given 2-component $Y$ is
 the 'best behaved'-- from $Y$'s viewpoint-- among those limits
 of the canonical
 system which are available in all cases, including
  the 'worst' case, when components across from $Y$ are hyperelliptic; or, what amounts to the same thing, this limit does not jump as $X$ varies, preserving $Y$. To make this precise, one has to turn the binode boundary loci into divisors, which
   in general involves the 'azimuthal
 modification' of a given family
  which we now define (see
  \cite{grd}, \S 6 for more details).
\par
To set things up, we give ourselves the following
\begin{itemize}
\item
$X/B$, a family of nodal curves, assumed versal near
the special fibre $X_0$. 
\item Then each bisep $\vtheta$ of $X_0$
corresponds to a regular codimension-2 subvariety
$\del_\vtheta\subset B$. 
\item Given a collection $\Theta_2$ of \emph{disjoint}
biseps, the corresponding loci $\del_\vtheta$ are evidently transverse.
Let $B(\Theta_2)$ denote their joint blowup in any order,
which is independent of the order by transversality. 
\item Moreover, we assume each $\vtheta\in\Theta_2$ is oriented, so there is a choice
of components $\lf{X}(\vtheta), \rt{X}(\vtheta)$ of $X_{\del_\vtheta}$. 
\item Then
 for each $\vtheta\in\Theta_2$, the inverse image of $\lf{X}(\vtheta)$ on the base-changed family $X_{B(\Theta_2)}$ is
a non-Cartier divisor. We let $X_{B(\Theta_2), L}$ be their
common blowup, called the \emph{left azimuthal modification}
associated to $\Theta_2$. Let $\lf{X}_{B(\Theta_2)}(\vtheta)$ be the associated Cartier divisor on $X_{B(\Theta_2), L}$ .
\end{itemize}

 With these preliminaries, the result is stated as follows .
\begin{thm}\lbl{sepcanonical-thm}
Let  the following be given:
\begin{itemize}\item
\[\pi: X\to B,\]  a family of nodal or smooth
curves with
irreducible general fibre, assumed versal near a fibre $X_0$;
\item
$\Theta_1$ , a set of seps,  $\Theta_2$ a set of maximal,
regularly middle-azimuthal biseps  on $X_0$; \item with $\Theta=\Theta_1\cup \Theta_2$,
a connected component $Y$ of the separation $X_0^\Theta$. Assume
 each $\theta_i\in \Theta_1$ and $\vtheta_i\in \Theta_2$
 is oriented having $Y$ on its
 left.
 \end{itemize}
Let \[\pi':X'=X_{L, B(\Theta_2)}\to B'=B(\Theta_2)\] be the associated left azimuthal modification  and $0'\in B'$
a preimage of 0 corresponding to the given collection of
regular middle azimuths $\md{\zeta_i}$  on $\Theta_2$.\par
Then:  a section
\[\alpha_0\in H^0(\omega_Y(3\sum\limits_{\lf{\theta_i}\in Y} \lf{\theta_i}+
 2\sum\limits_{\lf{\vtheta_i}\subset Y} (\lf{\vtheta_i})))\] extends near $0'\in B'$ to a local section
 \eqspl{omega'}
 {\alpha\in\pi'_*(\omega_{X'/B'}(-2\sum\limits_{\theta\in\Theta_1} \lf{ X'}(\theta)
 -\sum\limits_{\vtheta\in\Theta_2} \lf{X'}(\vtheta)))}
  if and only if \eqspl{alpha-in-sepcan}{\alpha_0\in
 |\omega_X|^\sep_\Theta|_Y.}
\end{thm}
\begin{proof} The proof is in essence an application of the
Residue Lemma \ref{residue-lem}.
We first prove sufficiency of \eqref{alpha-in-sepcan}, i.e. 'if'.
Denote the  line bundle in \eqref{omega'} on $X'$ by $\omega'$.
Because $\lf{X'}(\theta)+\rt{X'}(\theta)=\del_\theta$ for any
*-sep $\theta$, $\omega'$ is isomorphic locally over $B'$ to
\[\omega_{X'/B'}(+2\sum\limits_{\theta\in\Theta_1} \rt{ X'}(\theta)
 +\sum\limits_{\vtheta\in\Theta_2} \rt{X'}(\vtheta))\]
 Now the latter sheaf contains an analogous one where
$\Theta_1, \Theta_2$ are respectively replaced by
$\Theta_1^0, \Theta_2^0$ which  refer to the *-seps that meet
 $Y$.
Therefore, locally over $B'$,
$\omega'$ contains a subsheaf
\[\omega'_0\simeq \omega_{X'/B'}(-2\sum\limits_{\theta\in\Theta^0_1} \lf{ X'}(\theta)
 -\sum\limits_{\vtheta\in\Theta^0_2} \lf{X'}(\vtheta))\]
 Moreover, $\omega'_0$ is isomorphic to
 $\omega'$ near $Y$. Therefore in the proof of sufficiency, we shall henceforth assume
 all members of $\Theta_1, \Theta_2$ meet $Y$.\par
Next a remark.
Let $(Z,z)=(\rt{X_0}(\theta_i), \rt{\theta_i})$ for some $i$,
i.e. the side of $X_0$ across $\theta_i$ from $Y$.
Then $|\omega_Z|$ is always free at $z$ and $\omega_Z(-z)$ is
free at $z$ iff $(Z,z)$ is non-hyperelliptic. Similarly,
if $(Z,z_1,z_2)=(\rt{X_0}(\vtheta_i), \rt{\vtheta_i})$,
then the image of $|\omega_Z|$ in $\omega_Z\otimes\O_{z_1,z_2}$ is the hyperelliptic azimuth if $(Z,z_1,z_2)$ is hyperelliptic
and the whole 2-dimensional space otherwise. From this it is easy to see
that the azimuthal conditions, together with the lowered pole order condition ( 2 rather than 3) at
the hyperelliptic unimarks, cf. Definition \ref{septwist-order-def}, are necessary and sufficient
for $\alpha_0$ to extend to the entire special fibre
as a section of $\omega_{X_0}(\sum \rt{n}(\theta_i)\rt{X_0}(\theta_i)+2\sum\rt{X_0}(\vtheta_i)$.
Now we use an induction on $m=|\Theta|$. The case $m=1$ is trivial. For the induction step,
assume the conclusion is true for $m-1$ in place of $m$.
Pick any $\theta\in\Theta$ (sep or bisep), let
 and work over
  $B(\Theta\setminus\theta)$. Using Lemma \ref{residue-lem} for $Y\cup\lf{X'}(\theta)$, we
 can deform $\alpha_0$ to a differential $\alpha_1$
 on a component
  $Y_t$ of a nearby fibre $X'_t$ smoothing out
 exactly $\theta_1$ and keeping $\Theta\setminus\theta$
  intact, where $\alpha_1$
 satisfies similar residue conditions, $m-1$ in number,
 as well as appropriate azimuthal conditions
 (the fact that the latter can be imposed follows from an argument similar to the 'arbitrary residues summing to zero'
 argument in the proof of Lemma \ref{residue-lem}).
 Then by induction $\alpha_1$, hence $\alpha_0$, extends to the
 general fibre. This completes the proof of sufficiency.
\par We now prove necessity of \eqref{alpha-in-sepcan},
and accordingly drop the
 hypothesis that all members
 of $\Theta$ meet $Y$. We fix a bisep $\vtheta\subset Y$
that is right-hyperelliptic relative to $\Theta$ and establish that azimuthal and residue conditions for $\alpha_0$ at $\lf{\vtheta}$; other cases are similar or simpler.
\par Consider a *-sep
$\theta'\in \Theta$ on $\rt{X'}_{0'}(\theta)$
that is extremal relative to $Y$, i.e. such that there is no element of $\Theta$ strictly right of $\theta'$.
Thus $\rt{X'}_{0'}(\theta')$ is a component of $(\rt{X'}_{0'})^\Theta$. Let $X'_1$ be the unique component of
$(\rt{X'}_{0'})^\Theta$ containing the left side
 $\lf{\theta'}$.
Because our family is versal, there is
a nearby fibre $X'_t/B'$ where $\theta'$ deforms to
a bisep $\theta'_t$,  $(\rt{X'}_{0'}(\theta'), \rt{\theta'})$ survives intact,
while $\lf{X'}_{0'}(\theta'_t)$ smooths out. Because $\alpha$ extends to
this fibre, an argument as above shows that the residue and
azimuthal conditions hold at $\lf{\theta'}_t$. Hence by specialization, they also hold at $\lf{\theta'}$. Because
this is true for every extremal $\theta'$, it follows that
the restrictions  of $\alpha$ on
$X'_1$
belongs to the linear system $|\omega|^\sep(X'_1, \xi)$
 where $\xi$ is the azimuthal marking on $X'_1$ induced
 by the collection of the extemal $\theta'$. Now
 the latter system is composite with
the canonical system of $X'_1$, hence it fails to separate any
hyperelliptic pair on $X'_1$
and is ramified at every W-point (see Lemma \ref{composite-canonical-hypell-lem}).
Therefore we can continue and 'peel off' $X'_1$ just like
$\rt{X'}_{0'}(\theta')$: the azimuthal (and residue) conditions are satisfied on the left of every bisep on $X_1$ and similarly for seps. Continuing in this way, we eventually reach
$\vtheta$ on $Y$, so the conditions are satisfied there as well.

\end{proof}
\newsection{Semistable hyperelliptics: general case}
\label{ss-hypell-sec}
A stable, say semicompact-type, curve $X$ may be viewed as
 built up from 2-inseparable 'atoms'; as such, the canonical system
of $X$ behaves in an entirely \emph{heterogeneous} manner: it
can be birational on some
atoms but not on others. This is not so for the sepcanonical
system. Indeed a remarkable property of the
 sepcanonical system (on a middle-azimuthal stable curve), to be established here,
is that its mapping behavior
is quite \emph{homogeneous}: either it is 'essentially very ample'
(in particular birational on every atom); or else the curve is
hyperelliptic in the sense defined above, i.e. comprised
of 2-inseparable hyperelliptic atoms in a tree-like arrangement
(and in particular is of semicompact type), and
the sepcanonical system yields a 2:1 mapping on every atom.
In particular, different 2-components influence the behavior of the
sepcanonical system on each other. This property is part of
the main result of this section (Theorem \ref{azi-he-thm}).
\par
A  stable  middle-azimuthal curve is said to be
\emph{hyperelliptic} (as such) if it is hyperelliptic
 as ordinary curve and if moreover the given
 middle azimuth
 $\md{\zeta}(\vtheta)$
 at every bisep
$\vtheta$, necessarily hyperelliptic, is  the hyperelliptic
middle azimuth, i.e. the one
determined by the hyperelliptic left and
right azimuths at $\vtheta$.
   A semistable curve is
pseudo-hyperelliptic if its stable contraction is.
Our purpose here is to extend the characterization
of 2-inseperable hyperelliptics in terms of very ampleness
of the canonical system to the case of general nodal
 middle-azimuthal curves, where
the canonical system is replaced by the sepcanonical one.

We begin with an elementary Lemma on the effect of
azimuthal constraints on very ampleness.
\begin{lem}\lbl{hypell-twist-lem}
Let $Y$ be an inseparable
nodal curve, $p\neq q\in Y$
a pair of smooth points such that $Y$ is 2-inseparable
relative to $p+q$,
and $\zeta$
an azimuth at $(p,q)$. Then the $\zeta$-constrained subsystem of $|\omega_Y(2p+2q)|$
is essentially very ample unless $Y$
is 2-inseparable hyperelliptic,
$p+q$ is a hyperelliptic divisor
and $\zeta$ is the hyperelliptic azimuth.

\end{lem}
\begin{proof}

If $Y$ is not 2-inseparable, then $\omega_Y(p+q)$ is already essentially very ample by Lemma \ref{omega+2-insep-lem},
hence so is the constrained subsystem in question.
If $Y$ is 2-inseparable
 and $p+q$ is not a hyperelliptic divisor
 (e.g. $Y$ is not \hep),  $\omega_Y(p+q)$ is essentially
very ample by Lemma \ref{omega+2} and again we are done. \par
It remains to consider the case where $p+q$ is a hyperelliptic divisor on the 2-inseparable
hyperelliptic curve $Y$, where we must determine $\zeta$.
Then  the linear system $|\omega_Y(p+q)|$
is just $\sym^g\eta$ where $\eta= |p+q|$ is the $g^1_2$.
On the other hand, $|\omega(2p+2q)|$ is very ample and
contains as a hyperplane $\sym^{g+1}\eta$,
which induces the same map to $\P^1$ as $\eta$ (albeit as rational normal
curve of degree $g+1$). This hyperplane corresponds to a point $z_0$ on the secant line $\overline{pq}\subset \P^{g+2}=\P(H^0(\omega_Y(2p+2q))$ which can be identified with the hyperelliptic azimuth.
Because projection from $\overline{pq}$ itself corresponds to $\sym^g\eta$ (i.e. to the map to a rational normal curve
of degree $g$), a projection from any $z\in\overline{pq}$,
can fail to be an embedding off $p+q$ is if the projection
as a map on $X$ coincides with the $g^1_2$, i.e. coincides
with projection from $z_0$; but because $|\omega(2p+2q)|$ is very ample, this is only possible if $z=z_0$, i.e. the
hyperelliptic azimuth.
\end{proof}
An azimuthally marked curve $(Y,\xi)$ is said to be hyperelliptic if all unimarks in $\xi$ are
co-hyperelliptic and hyperelliptic
on $Y$ and all bimarks are co-hyperelliptic and hyperelliptic
on $Y$ and the associated azimuth in $\xi$ is the
hyperelliptic one. Then combining the above Lemma
with the Residue Lemma \ref{residue-lem}, we conclude
\begin{lem}\lbl{composite-canonical-hypell-lem}
Let $(Y, \xi)$ be a 2-inseparable
hyperelliptic azimuthally marked curve. Then
the sepcanonical system $|\omega|^\sep(Y, \xi)$ is
composite with the canonical system on $Y$.
\end{lem}
\begin{proof}
Part (i) of the Residue Lemma implies that we may assume
$\xi$ consists of a single (hyperelliptic) unimark or
azimuthal bimark. Then the bimark case follows from the
above Lemma, while the unimark case is straightforward.
\end{proof}
  Our result main result on the sepcanonical system,
   Theorem \ref{sepcanonical-thm} below,
   is that unless a stable  curve
is hyperelliptic, its sepcanonical system is 'essentially very ample'
(see below);
 only 'essentially' because of some exceptional behaviour at the
separators, but this turns out to be good enough for applications.
The proof has mostly been given above, in the discussion of
2-inseparables.
\par

\begin{defn}\lbl{eva-defn}
Let $Y$ be a curve with some unimarks and bimarks,
all smooth.
A linear system $(L,V)$ on $Y$ is said to be
\emph{essentially very ample} if it induces an embedding on every length-2 subscheme, except\par
(i) if $p$ is a unimark, at least one of \[(L(-ip), V(-ip)), V(-ip):=V\cap H^0(L(-ip)),
i=0, 1, 2 ,3\] induces an embedding on the subcheme $2p$;\par
(ii) if $(p,q)$ is a bimark, at least one of $(L(-ip-iq),V(-ip-iq)), i=0,1,2$
induces an embedding on $p+q$.
\end{defn}
Another consequence of Lemma \ref{hypell-twist-lem} is
\begin{lem}\lbl{ess-very-ample-2-comp-lem}
Let $\theta$ be a non-hyperelliptic azimuthal *-sep on $X$.
Assume $\theta$ is left-extremal, i.e. $\lf{X}(\theta)$
contains no other *-sep of $X$.
Then the sepcanonical system $|\omega_X|^\sep$ is essentially
very ample on the 2-component containing $\lf{\theta}$.
\end{lem}
\begin{proof}
Note $\lf{X}(\theta)$ is inseparable and 2-inseparable
relative to $\lf{\theta}$.
The cases where $\theta$ is \nhep as plain *-sep are settled
by Lemmas \ref{omega+3-on-2-sep-lem} and \ref{omega+2-insep-lem}. The remaining case, where $\theta$ is hyperelliptic but the azimuth is not, is settled by Lemma \ref{hypell-twist-lem}.
\end{proof}


\begin{thm}\lbl{azi-he-thm}
Let $X$ be a semistable middle-azimuthal curve of genus $g\geq 2$. Then
either\par (i) the sepcanonical system $|\omega_X|^{\sep}$ is
essentially very ample on each 2-component of $X$; or\par (ii) $X$
is hyperelliptic and $|\omega_X|^{\sep}$ maps each 2-
component of $X$ 2:1 to a rational normal curve.
\end{thm}
\begin{proof}First, it is easy to see that no generality is
lost by assuming $X$ is stable. By Proposition \ref{insep-hel}, we may assume $X$ is 2-separable.
We will assume to begin with that $X$ is of semicompact type
(Definition \ref{semicompact-defn}).\par
Assume  $X$ is  not azimuthally hyperelliptic.
Consider a 2-component $Y$ of $X$ that is an end-vertex
of the separation tree $G_2(X)$. We may assume $Y$ is left-extremal
for $X$.
 Let $\theta$ be the edge into $Y$. We assume
 $\theta$ is a bisep as the sep case is easier. By definition,
$\theta$ is not (azimuthally) hyperelliptic. Therefore by Lemma \ref{ess-very-ample-2-comp-lem}, $|\omega_X|^\sep$ is
essentially very ample on $Y$.
Now if $\theta$ is the unique edge of $G_2(X)$, we are done by
applying the same argument to its other vertex.
Else, there exists another end-vertex $Z$
with corresponding edge $\rho$, which we
may assume is right-extremal. By the same argument,
$|\omega_X|^\sep$ is
essentially very ample on $Z$. Because
$\theta\subset \lf{X}( \rho)$, $\lf{X}(\rho)$
is not hyperelliptic.
By induction on the number of components,
$|\omega_{\lf{X}(\rho)}|^\sep$ is \eva,
hence so is $|\omega_X|^\sep_{\lf{X}(\rho)}$.
Therefore $|\omega_X|^\sep$
is essentially very ample.\par
It remains to consider the case where $X$ is
 stable and not of semicompact type, hence contains a proper polyseparator $\Theta$. Such $X$ is never hyperelliptic.
The claim then is that $|\omega_X|^\sep$ is always essentially very ample.
Suppose first that $X$ contains a sep or maximal 2-sep, i.e. that the separation tree $G_2(X)$ is nontrivial.
Let $\theta$ be an edge. We may assume $\Theta\subset
\rt{X}(\theta)$, which implies $\rt{X}(\theta)$ is
\nhep. Replacing $\theta$ by a terminal edge to
its left, we may assume
$Y=\lf{X}(\theta)$ is an end-vertex of $G_2(X)$. Because
$\rt{X}(\theta)$
is \nhep, $|\omega_{\rt{X}(\theta)}|^\sep$
is essentially very ample by induction
on the number of components,
hence so is
$|\omega_X|^\sep|_{\rt{X}(\theta)}$. Now if $Y$ is \nhep,
$|\omega_X|^\sep$ is similarly essentially very ample on it.
Otherwise, $Y$ is hyperelliptic and in particular of semicompact type. Because, as an end-vertex,
$Y$ contains no
sep or maximal bisep of $X$, it follows that $Y$ is
2-inseparable relative to $\lf{\theta}$. Then it follows from
Lemma \ref{omega+3-on-2-sep-lem} that $|\omega_X|^\sep$ is \eva on $Y$.\par
Now we may assume $G_2(X)$ is trivial and $X$ contains
a proper polyseparator $\Theta$.
It will suffice to
prove that $|\omega_X|^\sep$ is essentially very ample
of each connected component $Y$ of $X^\Theta$. We
may assume
\[Y=\lf{X}(\vtheta), \vtheta=(\theta_1,\theta_2)
\subset \Theta=(\theta_1,...,\theta_n), n\geq 3.\]
If $Y$ is non-hyperelliptic, then $|\omega_Y|^\sep$
is already essentially very ample, hence so is the
bigger system  $|\omega_X|^\sep|_Y$. Therefore assume $Y$ is hyperelliptic. Note that $\vtheta$ is not right-
hyperelliptic, so there is no left azimuthal
condition attached to it. If $\vtheta$ is contained in a
single 2-component of $Y$ then
$Y$ is 2-inseparable, hence $|\omega_X|^\sep$ is \eva
on it by Lemma \ref{higher-twist-very-ample}.
 Else, each $\theta_i, i=1,2$ is contained in a different 2-component $Y_i$ of $Y$. Then because $Y$ contains no
 other *-seps of $X$, it follows that $Y$ must be 2-inseparable relative to $\lf{\theta_1}, \lf{\theta_2}$, so we can conclude by Lemma \ref{omega+3-on-2-sep-lem}.
\end{proof}
\begin{remarks}
(i) As stated, Theorem \ref{azi-he-thm} does not use Theorem \ref{sepcanonical-thm}. However
it is the latter Theorem that makes Theorem \ref{azi-he-thm} 'meaningful' by identifying 
the sepcanonical system with
a suitable limit.\par 
(ii) In \cite{grd}, we show that the hyperelliptic middle-azimuthal
stable curves are precisely the limits of smooth hyperelliptic curves. This result,
which of course does rely on Theorem \ref{sepcanonical-thm}, is proven by constructing a modification
of the Hodge bundle on a suitable parameter space (a blowup of the Hilbert scheme), together with a map to the relative canonical bundle, whose degeneracy locus meets the boundary transversely.
\end{remarks}

\bibliographystyle{amsplain}
\bibliography{mybib}

\end{document}